\documentclass[11pt]{amsart}
\usepackage[utf8]{inputenc}
\usepackage{amsmath}
\usepackage{amsthm, amsfonts,srcltx,amsopn}
\usepackage{mathabx}
\usepackage{amssymb}
\usepackage[mathscr]{euscript}
\usepackage{enumerate}

\usepackage{paralist}

\usepackage{graphicx}
\usepackage[margin=1in]{geometry}

\usepackage{tikz}
\usepackage{color}

\definecolor{Green}{RGB}{30, 150, 30}

\usepackage{mathrsfs}
\usepackage{hyperref}
\usepackage{setspace}

\usetikzlibrary{arrows}
\usepackage{tikz-cd} 
\tikzcdset{row sep/normal=.1cm}
\tikzcdset{column sep/normal=.75cm}
\newtheorem{thm}{Theorem}[section]
\newtheorem{claim}[thm]{Claim}
\newtheorem{prop}[thm]{Proposition}
\newtheorem{lem}[thm]{Lemma}
\newtheorem{cor}[thm]{Corollary}

\theoremstyle{definition}
\newtheorem{defn}[thm]{Definition}

\newtheorem{rem}[thm]{Remark}

\newcommand*\supp{\operatorname{supp}}

\newcommand*\cusp{\operatorname{cusp}}
\newcommand*\gate{\mathfrak{g}}
\newcommand*\stab{\operatorname{Stab}}

\newcommand*\diam{\operatorname{diam}}

\newcommand*\nest{\sqsubseteq}
\newcommand*\propnest{\sqsubsetneq}

\newcommand*\mf[1]{\mathfrak {#1}}
\newcommand*\mc[1]{\mathcal {#1}}
\newcommand*\trans{\pitchfork}

\def\P{\mathbf{P}}

\newcommand\coset[1]{P(#1)}

\newcommand*\gp[2]{({#1} \mid {#2})_{x_0}}
\newcommand*\gprel[3]{({#1} \mid {#2})_{#3}}

\newcounter{list}

\def\F{\mathbf{F}}

\newcounter{jcomments}

\newcounter{jrcomments}

\newcounter{ccomments}

\newcommand{\s}{\mf{S}}

\title[Relative hyperbolicity, thickness, and the hierarchically 
hyperbolic boundary]{Relative hyperbolicity, thickness,\\ and the hierarchically hyperbolic boundary}

\author{Carolyn Abbott}
\address{Brandeis University, Waltham, MA, USA}
\email{carolynabbott@brandeis.edu}

\author{Jason Behrstock}
\address{Lehman College and The Graduate Center, CUNY, New York, New York, USA}
\email{jason.behrstock@lehman.cuny.edu}

\author{Jacob Russell}
	\address{Department of Mathematics, Rice University, Houston, TX}
\email{jacob.russell@rice.edu}

\begin{document}

\begin{abstract}We study the boundaries of relatively hyperbolic HHGs.
Using the simplicial structure on the hierarchically hyperbolic
boundary, we characterize both relative hyperbolicity and being thick
of order 1 among HHGs.  In the case of relatively hyperbolic HHGs, we
show that the Bowditch boundary of the group is the quotient of
the HHS boundary obtained by collapsing the limit sets of the
peripheral subgroups to a point.  In establishing this, we give a
construction that allows one to modify an HHG structure by including 
a collection of hyperbolically
embedded subgroups into the HHG structure.\end{abstract}

\maketitle

\section{Introduction}
Boundaries play a central role in  the coarse geometry of groups and spaces exhibiting aspects of non-positive curvature. For example, the dynamics of the action on the Gromov boundary and the Bowditch boundary completely characterize hyperbolic and relative hyperbolic groups respectively \cite{bowditch_uniform_hyperbolicity,Yaman_rel_hyp_dynamics}. Moreover, the quasi-conformal  structure of  these boundaries completely determines the coarse geometry of these groups \cite{paulin,bourdon_quasiconformal,bonk_schramm_quasiconformal}. CAT(0) groups,  particularly cubulated  groups, have a variety of different boundaries that capture different aspects of the geometry of these groups at infinity; see, e.g., \cite{Hagen_simplicial_boundary,Beyrer_Fioravanti_cross_ratio,Mihalik_Ruane_non-locally_connected}.

In this paper, we examine the connection between the boundary of
\emph{hierarchically hyperbolic groups} and relative hyperbolicity.
Hierarchical hyperbolicity is a coarse notion of non-positive
curvature introduced by Behrstock, Hagen, and Sisto, which is
enjoyed by a large number of groups including mapping class groups,
virtually special groups, most $3$--manifold groups, and extra
large type Artin groups
\cite{BHS_HHSI,BHS_HHSII,HagenMartinSisto:XLartinHHS}.  The main idea
behind hierarchical hyperbolicity is that the geometry of the group
$G$ can be well understood via a collection of projection maps
$\mf{S}=\{\pi_W \colon G \to \mc{C} W\}$ of the group onto various
hyperbolic spaces $\mc{C}W$.

Durham, Hagen, and Sisto introduce a boundary for hierarchically
hyperbolic groups \cite{HHS_Boundary}.  The boundary combines the
Gromov boundaries, $\{\partial \mc{C}W\}$, of the various hyperbolic
spaces into a simplicial complex---denoted
$\partial_\Delta(G,\mf{S})$---that captures naturally occurring
product regions in the group.  This simplicial structure is analogous
to the simplicial boundary of a CAT(0) cube complex introduced by
Hagen \cite{Hagen:quasi_arboreal}.  

Our first result uses this
simplicial complex to characterize when a hierarchically hyperbolic
group is relatively hyperbolic. This type of result has a long 
history. One of the first such results was by Hruska and Kleiner who 
proved a classification for CAT(0) spaces with isolated flats 
\cite{Hruska_Kleiner_Isolated_Flats};   
our formulation below is a direct analogue of a result
of Behrstock and Hagen characterizing relative hyperbolicity in
cubical groups using the simplicial boundary
\cite{BehrstockHagen:cubulated1}.

\begin{thm}\label{thm:intro_rel_hyp+boundary}
	Let $(G,\mf{S})$ be a hierarchically hyperbolic group. The group $G$ is hyperbolic relative to a collection of infinite index subgroups $\{H_1,\dots, H_k\}$ if and only if   each $H_i$ is hierarchically quasiconvex and there is a  collection  $\{\Lambda_1,\dots, \Lambda_k\}$  of subcomplexes of $\partial_\Delta(G,\mf{S})$ so that 
	\begin{enumerate}
		\item each $\Lambda_i$ is the limit set of $H_i$;
		\item any two translates $g\Lambda_i$ and $h\Lambda_j$ are either disjoint or equal;
		\item the complement of the orbit of $\Lambda_1 \cup \dots \cup \Lambda_k$ is  a non-empty set of isolated vertices.
	\end{enumerate}
\end{thm}

Complementary to Theorem \ref{thm:intro_rel_hyp+boundary}, we use the
simplicial structure on the boundary to understand when a
hierarchically hyperbolic group is \emph{thick of order 0 or 1} relative to
hierarchically quasiconvex subsets.  Thickness is a powerful
obstruction to relative hyperbolicity that can also provide upper
bounds on the divergence of a space.

\begin{thm}\label{thm:intro_thickness}
	Let $(G,\mf{S})$ be a hierarchically hyperbolic group.
	\begin{enumerate}
		\item   A hierarchically quasiconvex subgroup of $G$ is thick of order 0 if and only if its limit set in $\partial_\Delta(G,\mf{S})$ is a join.
		\item If $G$ is thick of order 1 with respect to a finite collection of hierarchically quasiconvex subgroups, then $\partial_\Delta(G,\mf{S})$ is disconnected and contains a positive dimensional $G$--invariant connected component.
		\item If $\partial_\Delta(G,\mf{S})$ is disconnected and contains a positive dimensional $G$--invariant connected component, then $G$ is  thick of order $1$ with respect to hierarchically quasiconvex subsets. 
	\end{enumerate}
\end{thm}

In addition to the simplicial structure, Durham, Hagen, and Sisto
equip the hierarchically hyperbolic boundary with a more sophisticated
topology.  Using this topology, we show that the
Bowditch boundary of a relatively hyperbolic HHG is a natural quotient
of the hierarchically hyperbolic boundary.  Analogous results have
been shown for relatively hyperbolic CAT(0) groups by Tran \cite{Tran_relation_between_boundaries} and for 
relatively hyperbolic structures on hyperbolic groups by Spriano \cite{Spriano_hyperbolic_I}  and Manning \cite{Manning_bowditch_boundary}.

\begin{thm}\label{thm:intro_quotient}
	Let $G$ be a hierarchically hyperbolic group that is hyperbolic relative to a finite collection of subgroups $\mc{P}$. The Bowditch boundary of the $G$ relative to $\mc{P}$ is the quotient of the HHS boundary of $G$ obtained by collapsing the limit set of each coset of a peripheral subgroup to a point.
\end{thm}

A particularly interesting case of Theorem \ref{thm:intro_quotient} is the case of a closed, irreducible, non-geometric  $3$--manifold with at least one hyperbolic piece in its JSJ decomposition. The fundamental group of such a manifold is hyperbolic relative to the fundamental groups of the maximal tori and graph manifold pieces. While these groups are not always CAT(0), they are always hierarchically hyperbolic \cite{BHS_HHSII,HRSS_3-manifolds}. Hence we have the following. 
\begin{cor}
	Let $M$ be an irreducible, non-geometric closed $3$--manifold with at least one hyperbolic piece in  its JSJ decomposition. Let $N_1,\dots,N_k$ be the maximal graph manifold and tori pieces of the JSJ decomposition. The Bowditch boundary of $\pi_1(M)$ relative to $\pi_1(N_1),\dots, \pi_1(N_k)$ is the quotient of the hierarchically hyperbolic boundary of $\pi_1(M)$ obtained by collapsing the limit set of each coset of each $\pi_1(N_i)$ to a point.
\end{cor}

Both Theorem \ref{thm:intro_rel_hyp+boundary} and Theorem
\ref{thm:intro_quotient} are facilitated by a pair of technical
results that allow us to ensure compatibility of the relatively
hyperbolic and hierarchically hyperbolic structures we are considering
on our group.  The first is our previous work in
\cite{ABR:maximization}, which shows that performing a particular
``maximization procedure'' on the projection structure of a
hierarchically hyperbolic group does not change the simplicial or
topological structure of the boundary.  The second is the following
result, which shows that given a hierarchically hyperbolic group 
one can augment the hierarchically hyperbolic structure by adding in 
any hyperbolically embedded subgroup; see
Section \ref{sec:add_hyp_embedded} for a more precise statement.

\begin{thm}\label{thm:intro_hyp_embed}
	Let $G$ be a hierarchically hyperbolic group and
	$\{H_1,\dots,H_k\}$ be a hyperbolically embedded collection of
	subgroups of $G$.  There exists a hierarchically hyperbolic
	structure for $G$ so that the cosets of the $H_i$ index hyperbolic
	spaces whose associated product regions are the cosets of the
	$H_i$.
\end{thm}

Readers familiar with hierarchically hyperbolic groups will know that
every hierarchically hyperbolic group admits many different
hierarchically hyperbolic structures.  It is an open question whether
or not different hierarchically hyperbolic structures produce
topologically distinct boundaries.  However, as a consequence of our
work in \cite{ABR:maximization}, Theorems
\ref{thm:intro_rel_hyp+boundary}, \ref{thm:intro_thickness}, and
\ref{thm:intro_quotient} all apply regardless of which 
hierarchically hyperbolic structure is being considered.

\subsection{Organization of the paper}
In Section \ref{background}, we define relatively hyperbolic and
hierarchically hyperbolic spaces and collect some result from the
literature.  In Section \ref{sec:add_hyp_embedded}, we prove our main
technical tool (Theorem \ref{thm:intro_hyp_embed}) showing that
hyperbolically embedded subgroups can be added into a hierarchically
hyperbolic structure.  Sections \ref{sec:boundary_rel_hyp} and
\ref{sec:bowditch_boundary} establish our theorems
on the HHS boundary of relatively hyperbolic group.  In Section
\ref{sec:boundary_rel_hyp}, we characterize relative hyperbolicity via
the simplicial structure on the hierarchically hyperbolic boundary
(Theorem \ref{thm:intro_rel_hyp+boundary}), and in Section
\ref{sec:bowditch_boundary}, we show the Bowditch boundary is a
quotient of the HHS boundary (Theorem \ref{thm:intro_quotient}).  In
Section \ref{sec:thick_case}, we recall the notion of a thick metric
space and establish the connection between the HHS boundary and being
thick of order 0 or 1 (Theorem~\ref{thm:intro_thickness}).

\subsection*{Acknowledgments}
We thank Davide Spriano for explaining how to apply the results of
\cite{PS_Unbounded_domains} in the setting of Section
\ref{sec:thick_case}.

Abbott was supported by NSF grants DMS-1803368 and 
DMS-2106906. Behrstock was supported by the Simons Foundation as a 
Simons Fellow. Behrstock thanks the Barnard/Columbia Mathematics 
department for their hospitality. Russell was supported by NSF grant DMS-2103191.

\section{Background on hierarchical  and relative hyperbolicity}\label{background}

\subsection{Coarse Geometry}
Let $(X,d_X)$ be a metric space. For $Y\subseteq X$ and any constant $C\geq 0$,  we denote the closed $C$--neighborhood of $Y$ in $X$ by
\[\mathcal N_C(Y)=\{x\in X : d_X(x,Y)\leq C\}.\]     Two  subsets $Y,Z \subseteq X$ are \emph{$C$--coarsely equal}, for some $C\geq 0$, if $Y \subseteq \mc{N}_C(Z)$ and $Z  \subseteq \mc{N}_C(Y)$. 
When $Y$ and $Z$ are $C$--coarsely equal, we write $Y\asymp_C Z$.  
 
A function $f \colon X \to 2^{Y}$ is a $C$--\emph{coarse map} if $f(x)$ is a non-empty set of diameter at most $C$  for all $x \in X$. The $C$--coarse map $f \colon  X \to 2^{Y}$ is $C$--\emph{coarsely onto} if $Y \subseteq \mc{N}_C(f(X))$. 

A \emph{$(\lambda,\varepsilon)$--quasi-geodesic} is a
$(\lambda,\varepsilon)$--quasi-isometric embedding of a closed interval
$I\subseteq \mathbb R$ into $X$, and a \emph{geodesic} is an isometric
embedding of $I$ into $X$.   In the case of quasi-geodesics, we allow $f$ to be a coarse map.

A (coarse) map $f\colon
[0,T]\to X$ is an \emph{unparametrized $(\lambda,\varepsilon)$--quasi-geodesic}
if there exists a non-decreasing function $g\colon
[0,T']\to[0,T]$ such that the following hold:
\begin{itemize} 
	\item $g(0)=0$, 
	\item $g(T')=T$,  
	\item $f\circ g\colon [0,T']\to X$ is a $(\lambda,\varepsilon)$--quasi-geodesic, and	
	\item for each $j\in[0,T']\cap \mathbb N$,  the diameter of
	$f(g(j))\cup f(g(j+1))$ is at most $\varepsilon$.  
\end{itemize}

A  geodesic metric space $X$ is \emph{$\delta$--hyperbolic} if  any geodesic triangle with sides $\gamma_1,\gamma_2,\gamma_3$ satisfies $\gamma_3 \subseteq \mc{N}_\delta(\gamma_1\cup \gamma_2)$. A subset $Y$ of a $\delta$--hyperbolic space $X$ is \emph{$\mu$--quasiconvex} if every geodesic in $X$ between points in $Y$ is contained in $\mc{N}_\mu(Y)$.  When $Y$ is $\mu$--quasiconvex, there is a well defined, coarsely Lipschitz coarse map $\mf{p}_Y \colon X \to Y$ with constants depending only on $\delta$ and $\mu$ so
that $$\mf{p}_{Y}(x) = \{y \in Y : d_X(x,y) \leq d_X(x,Y) +1\}.$$ We call the map
$\mf{p}_{Y}$ the \emph{closest point projection onto $Y$}. 

Given any subset $Y$ of a $\delta$--hyperbolic space $X$, the \emph{convex hull} $H(Y)$ of $Y$ is the union of all geodesics between pairs of points in $Y$. For any subset, the convex hull is $\mu$--quasiconvex for some $\mu$ depending only on $\delta$. If $Y$ is itself $\mu'$--quasiconvex, then $Y$ and $H(Y)$ are coarsely equal with constant depending only on $\delta$ and $\mu'$. In this case,  $\mf{p}_{Y}(x)$ is uniformly coarsely equal to $\mf{p}_{H(Y)}(x)$ for all $x \in X$, and the path metric on $H(Y)$ is a geodesic metric that is quasi-isometric to the subset metric on $Y$.

\subsection{Hierarchical hyperbolicity}

\begin{defn}[HHS]\label{defn:HHS}
	Let $E>0$ and $\mc{X}$ be an $(E,E)$--quasi-geodesic space.  A
	\emph{hierarchically hyperbolic space (HHS) structure with constant $E$}
	for $\mc{X}$ is an index set $\mathfrak S$ and a set $\{ \mc{C}W :
	W\in\mathfrak S\}$ of $E$--hyperbolic spaces $(\mc{C}W,d_W)$ such
	that the following axioms are satisfied.  \begin{enumerate}
		
		\item\textbf{(Projections.)}\label{axiom:projections} For each $W \in \mf{S}$, there exists a \emph{projection} $\pi_W \colon \mc{X} \rightarrow 2^{\mc{C}W}$  that is a $(E,E)$--coarsely Lipschitz, $E$--coarsely onto, $E$--coarse map.

		\item \textbf{(Nesting.)} \label{axiom:nesting} If $\mathfrak S \neq \emptyset$, then $\mf{S}$ is equipped with a  partial order $\nest$ and contains a unique $\nest$--maximal element. When $V\nest W$, we say $V$ is \emph{nested} in $W$.  For each
		$W\in\mathfrak S$, we denote by $\mathfrak S_W$ the set of all $V\in\mathfrak S$ with $V\nest W$.  Moreover, for all $V,W\in\mathfrak S$ with $V\propnest W$ there is a specified non-empty subset $\rho^V_W\subseteq \mc{C}W$ with $\diam(\rho^V_W)\leq E$.

		\item \textbf{(Orthogonality.)} 
		\label{axiom:orthogonal} $\mathfrak S$ has a symmetric relation called \emph{orthogonality}. If $V$ and $W$ are orthogonal, we write $V\perp
		W$ and require that $V$ and $W$ are not $\nest$--comparable. Further, whenever $V\nest W$ and $W\perp
		U$, we require that $V\perp U$. We denote by $\mf{S}_W^\perp$ the set of all $V\in \mf{S}$ with $V\perp W$.

		\item \textbf{(Transversality.)}
		\label{axiom:transversality} If $V,W\in\mathfrak S$ are not
		orthogonal and neither is nested in the other, then we say $V$ and $W$ are
		\emph{transverse}, denoted $V\trans W$.  Moreover, for all $V,W \in \mathfrak{S}$ with $V\trans W$ there are non-empty
		sets $\rho^V_W\subseteq \mc{C}W$ and
		$\rho^W_V\subseteq \mc{C} V$ each of diameter at most $E$.

		\item \textbf{(Finite complexity.)} \label{axiom:finite_complexity} Any set of pairwise $\nest$--comparable elements has cardinality at most $E$.
		
		\item \textbf{(Containers.)} \label{axiom:containers}  For each $W \in \mf{S}$ and $U \in \mf{S}_W$ with $ \mf{S}_W\cap \mf{S}_U^\perp \neq \emptyset$, there exists $Q \in\mf{S}_W$ such that $V \nest Q$ whenever $V \in\mf{S}_W \cap \mf{S}_U^\perp$.  We call $Q$ the \emph{container of $U$ in $W$}.
		
		\item\textbf{(Uniqueness.)} There exists a function 
		$\theta \colon [0,\infty) \to [0,\infty)$ so that for all $r \geq 0$, if $x,y\in\mc X$ and
		$d_\mc{X}(x,y)\geq\theta(r)$, then there exists $W\in\mathfrak S$ such
		that $d_W(\pi_W(x),\pi_W(y))\geq r$. \label{axiom:uniqueness}
		
		\item \textbf{(Bounded geodesic image.)} \label{axiom:bounded_geodesic_image} 
		For all  $V,W\in\mathfrak S$ and for all $x,y \in \mc{X}$, if  $V\propnest W$ and $d_V(\pi_V(x),\pi_V(y)) \geq E$, then every $\mc{C}W$--geodesic from $\pi_W(x)$ to $\pi_W(y)$ must intersect $\mc{N}_E(\rho_W^V)$.

		\item \textbf{(Large links.)} \label{axiom:large_link_lemma} 
		For all $W\in\mathfrak S$ and  $x,y\in\mc X$, there exists $\{V_1,\dots,V_m\}\subseteq\mathfrak S_W -\{W\}$ such that $m$ is at most $E d_{W}(\pi_W(x),\pi_W(y))+E$, and for all $U\in\mathfrak
		S_W - \{W\}$, either $U\in\mathfrak S_{V_i}$ for some $i$, or $d_{U}(\pi_U(x),\pi_U(y)) \leq E$.  
		
		\item \textbf{(Consistency.)}
		\label{axiom:consistency} For all $x \in\mc X$ and $V,W,U \in\mf{S}$:
		\begin{itemize}
			\item  if $V\trans W$, then $\min\left\{d_{W}(\pi_W(x),\rho^V_W),d_{V}(\pi_V(x),\rho^W_V)\right\}\leq E$,
			\item if $U\nest V$ and either $V\propnest W$ or $V\trans W$ and $W\not\perp U$, then $d_W(\rho^U_W,\rho^V_W)\leq E$.
		\end{itemize}

		\item \textbf{(Partial realization.)} \label{axiom:partial_realisation}  If $\{V_i\}$ is a finite collection of pairwise orthogonal elements of $\mathfrak S$ and $p_i\in  \mc{C}V_i$ for each $i$, then there exists $x\in \mc X$ so that:
		\begin{itemize}
			\item $d_{V_i}(\pi_{V_{i}}(x),p_i)\leq E$ for all $i$;
			\item for each $i$ and 
			each $W\in\mathfrak S$, if $V_i\propnest W$ or $W\trans V_i$, we have 
			$d_{W}(\pi_W(x),\rho^{V_i}_W)\leq E$.
		\end{itemize}
	\end{enumerate}

	We use $\mf{S}$ to denote the hierarchically hyperbolic space structure, including the index set $\mf{S}$, spaces $\{\mc{C}W : W \in \mf{S}\}$, projections $\{\pi_W : W \in \mf{S}\}$, and relations $\nest$, $\perp$, $\trans$. We call the elements of $\mf{S}$ the \emph{domains} of $\mf{S}$ and call the maps $\rho_W^V$ the \emph{relative projections} from $V$ to $W$. The number $E$ is called the \emph{hierarchy constant} for $\mf{S}$.
	
	A quasi-geodesic space $\mc{X}$ is a  \emph{hierarchically 
		hyperbolic space with constant $E$} if there exists a 
	hierarchically hyperbolic structure on $\mc{X}$ with constant $E$. The pair $(\mc{X},\mf{S})$  denotes a  hierarchically hyperbolic space equipped with the specific  HHS structure $\mf{S}$.
\end{defn}	

When writing the distances in the hyperbolic spaces $\mc{C}W$ between images of points under $\pi_W$, we will frequently suppress the $\pi_W$ notation. That is, we will use $d_W(x,y)$ to denote $d_{W}(\pi_W(x),\pi_W(y))$ for $x,y\in\mc{X}$. 

When two domains are nested, $V \propnest W$, the above axioms only
require an ``upward'' \emph{relative projection} $\rho^V_W$.  However,
the coarse surjectivity of the projection maps plus the bounded
geodesic image axiom allows us to define a ``downward'' relative
projection that is well behaved away from the $\rho_W^V$.  This
downward relative projection is used in defining the topology on the
HHS boundary in Section \ref{sec:bowditch_boundary}.

\begin{lem}[{\cite[Proposition 1.11]{BHS_HHSII}}]\label{lem:downward_rhos} 
	Let $(\mc{X},\mf{S})$ be a hierarchically hyperbolic space with constant $E$. For all $W,V \in \mf{S}$ with $V \propnest W$, there exists a map $\rho_V^W \colon \mc{C}W \to \mc{C}V$ and a constant $E' \geq 0$, depending only on $E$, so that 
	\begin{itemize}
		\item if a $\mc{C} W$--geodesic $\gamma$ does not intersect $\mc{N}_{E'}(\rho_W^V)$, then $\diam_{\mc{C}V}(\rho_V^W(\gamma)) \leq E'$; and
		\item   for all $x \in \mc{X}$, $\min\left\{ d_W(\pi_W(x), \rho_V^W), \diam\bigl(\pi_V(x) \cup \rho_W^V(\pi_V(x))\bigr) \right\} \leq E'$.
	\end{itemize}
\end{lem}

For a hierarchically hyperbolic space $(\mc{X},\mf{S})$, we are often most concerned with the domains $W \in \mf{S}$ whose associated hyperbolic spaces $\mc{C}W$ have infinite diameter. Hence, we often also restrict to HHSs with the following regularity condition.

\begin{defn}[Bounded domain dichotomy]
	Given an HHS $(\mc{X},\mf{S})$, we let $\mf{S}^\infty$ denote the
	set $\{W \in\mf{S}: \diam(\mc{C}W) = \infty\}$.  We refer to the
	domains in $\mf{S}^\infty$ as \emph{unbounded domains} and the
	domains not in $\mf{S}^\infty$ as \emph{bounded domains}.  We say
	that $(\mc{X},\mf{S})$ has the \emph{bounded domain dichotomy} if 
	the diameter of 
	each $\mc{C}W$ is either infinite or uniformly bounded, i.e., 
	there is some $D \geq 0$ such that for all $W\in\mf{S} -
	\mf{S}^{\infty}$ we have $\diam(\mc{C}W) \leq D$.
\end{defn}

The bounded domain dichotomy is a natural condition as it is satisfied
by all \emph{hierarchically hyperbolic groups (HHG)}, which is a condition 
requiring equivariance of the HHS structure.  In this paper, we work 
with a class of finitely generated groups that is slightly more 
general than being an HHG (see Remark~\ref{rem:HHGvsNHHG}); these are groups that have an HHS structure compatible
with the action of the group in the following way.

\begin{defn}[$G$--HHS]\label{defn:nearlyHHG}
	Let $G$ be a finitely generated group.  A hierarchically
	hyperbolic space $(\mc{X},\mf{S})$ with constant $E$ that has the bounded domain dichotomy is a 
	\emph{$G$--HHS} if the following hold.
	\begin{enumerate}
		\item $\mc{X}$ is a proper metric space with a proper and cocompact action of $G$ by isometries.
		\item  $G$ acts on $\mf{S}$ by a $\nest$--, $\perp$--, and $\trans$--preserving bijection, and $\mf S^\infty$ has finitely many $G$--orbits.
		\item For each $W \in \mf{S}$ and $g\in G$, there exists an isometry $g_W \colon \mc{C}W \rightarrow \mc{C}gW$ satisfying the following for all $V,W \in \mf{S}$ and $g,h \in G$.
		\begin{itemize}
			\item The map $(gh)_W \colon \mc{C}W  \to \mc{C}ghW$ is equal to the map $g_{hW} \circ h_W \colon \mc{C}W \to \mc{C}ghW$.
			\item For each $x \in \mc{X}$, $g_W(\pi_W(x)) \asymp_E \pi_{gW}(g \cdot x)$.
			\item If $V \trans W$ or $V \propnest W$, then $g_W(\rho_W^V)  \asymp_E \rho_{gW}^{gV}$.
		\end{itemize}
	\end{enumerate}
	We can and will assume that $\mc X$ is $G$ equipped with a finitely 
	generated word metric. We say that $\mf{S}$ is a \emph{$G$--HHS structure} for  the group $G$ and use the pair $(G,\mf{S})$ to denote the group $G$ equipped with the specific $G$--HHS structure $\mf{S}$.
\end{defn} 

\begin{rem}[$G$--HHS versus HHG]\label{rem:HHGvsNHHG}
	The difference between the above definition of a $G$--HHS and a hierarchically hyperbolic group (HHG) is that a hierarchically hyperbolic group is required to act with finitely many orbits on $\mf{S}$ instead of $\mf{S}^\infty$. In particular, each HHG is also a $G$--HHS. As the definition of the hierarchically hyperbolic boundary does not involve the uniformly bounded diameter domains, it is natural for us to work in the slightly more general $G$--HHS setting. Moreover, many of our arguments will rely upon a ``maximization procedure'' introduced in \cite{ABD} to transform a given hierarchically hyperbolic structure into one with desirable properties; see Section \ref{subsec:maximiztion}. The  maximization procedure introduces a large number of uniformly bounded domains into the HHS structure, and the result of maximizing an HHG is a $G$--HHS and not necessarily an HHG. Working with $G$--HHSs from the outset is therefore simpler as they are closed under this maximization procedure.
\end{rem}

One of the most prominent features of hierarchically hyperbolic spaces is that every pair of points can be joined by a \emph{hierarchy path}---a quasi-geodesic that projects to an unparametrized quasi-geodesic in each hyperbolic space $\mc{C}W$.

\begin{defn}
	A \emph{$\lambda$--hierarchy path} in a hierarchically hyperbolic space $(\mc{X},\mf{S})$ is a $(\lambda,\lambda)$--quasi-geodesic $\gamma$ in  $\mc{X}$ so that   $\pi_W \circ \gamma$ is an unparametrized $(\lambda,\lambda)$--quasi-geodesic for all $W \in \mf{S}$.
\end{defn}

\begin{thm}[{\cite[Theorem~4.4]{BHS_HHSII}}]
	For all $E \geq0$, there exists $\lambda \geq 1$ so that every pair of points in a hierarchically hyperbolic space with constant $E$ is joined by a $\lambda$--hierarchy path.
\end{thm}

\subsection{Hierarchical quasiconvexity and standard product regions}

The analogue of quasiconvex subsets of a hyperbolic space in the
setting of hierarchical hyperbolicity are the following
\emph{hierarchically quasiconvex} subsets. We refer the reader to 
\cite[Section 5]{BHS_HHSII} for details on any of the background material 
in this subsection.

\begin{defn}
	Let $k \colon [0,\infty) \to [0,\infty)$. A subset $\mc{Y}$ of an HHS $(\mc{X},\mf{S})$ is \emph{$k$--hierarchically quasiconvex} if 
	\begin{enumerate}
		\item  $\pi_W(\mc{Y})$ is a $k(0)$--quasiconvex subset of $\mc{C}W$ for each $W \in \mf{S}$; and
		\item if $x \in \mc{X}$ satisfies $d_W(x,\mc{Y}) \leq r$ for each $W\in\mf{S}$, then $d_{\mc{X}}(x,\mc{Y}) \leq k(r)$.
	\end{enumerate}
	A subgroup $H$ of a $G$--HHS $(G,\mf{S})$ is \emph{hierarchically quasiconvex} if $H$ is a hierarchically quasiconvex subset of $G$ equipped with a finitely generated word metric.
\end{defn}

Whether or not a subset is hierarchically quasiconvex can depend on which HHS structure is put on the space, hence $\mc{Y}$ is a hierarchically quasiconvex subset of $(\mc{X},\mf{S})$ and not just $\mc{X}$.

 Hierarchical quasiconvexity is equivalent to the property that every hierarchy path with endpoints on the subset stays uniformly close to the subset.
 
 \begin{prop}[{\cite[Proposition 5.7]{RST_Quasiconvexity}}]\label{prop:HQC_and_hp}
 	A subset $\mc{Y}$ of an HHS $(\mc{X},\mf{S})$ is $k$--hierarchically quasiconvex if and only if there is a function $Q \colon [0,\infty) \to [0,\infty)$ so that for each $\lambda \geq 1$, every $\lambda$--hierarchy path with end points on $\mc{Y}$ is contained in the $Q(\lambda)$--neighborhood of $\mc{Y}$. Moreover, the functions $k$ and $Q$ each determine the other.
 \end{prop}
 
 Proposition \ref{prop:HQC_and_hp} implies that the definition of a hierarchically quasiconvex subgroup is independent of the choice of finite generating set for the ambient group.  Moreover, by mimicking the proofs in the case of quasiconvex subgroups of hyperbolic groups (with hierarchy paths replacing geodesics), we have that hierarchically quasiconvex subgroups are finitely generated and undistorted.

\begin{lem}\label{lem:HQC_undistorted}
	Let $(G,\mf{S})$ be a $G$--HHS. If $H<G$ is hierarchically quasiconvex, then $H$ is finitely generated and undistorted.
\end{lem}

Each hierarchically quasiconvex subset $\mc Y$ comes equipped with a \emph{gate map} denoted $\gate_{\mc Y} \colon \mc{X} \to \mc{Y}$. While this map might not be the coarse closest point projection, it has a number of nice properties that we summarize below.

\begin{lem}[{\cite[Lemma 5.5]{BHS_HHSII}}]\label{lem:hierarch_path_through_the_gate}
	Let $(\mc{X},\mf{S})$ be an HHS with constant $E$. Suppose $\mc{Y} \subseteq \mc{X}$ is $k$--hierarchically quasiconvex. There is a coarse map $\gate_{\mc{Y}}\colon \mc{X} \to \mc{Y}$ and a constant $\kappa \geq 1$ depending only on $k$ and $E$, so that the following hold.
	\begin{itemize}
		\item For all $y \in \mc{Y}$, we have  $d_\mc{X}(y, \gate_{\mc{Y}}(y)) \leq \kappa$. 
		\item The map $\gate_{\mc{Y}}$ is $(\kappa,\kappa)$--coarsely Lipschitz.
		\item For each $x \in \mc{X}$ and $W \in \mf{S}$, we have \[\pi_W(\gate_{\mc{Y}}(x)) \asymp_\kappa \mf{p}_{\pi_W(\mc{Y})}(\pi_W(x)).\]
	\end{itemize}
\end{lem}

Each domain in an hierarchically hyperbolic space has an associated hierarchically quasiconvex subset $\P_W$: 

\begin{defn}\label{defn:product_region}
	Let $(\mc{X},\mf{S})$ be an hierarchically hyperbolic space with constant $E$. For each $W \in \mf{S}$, define the \emph{standard product region} for $W$ to be the set \[ \P_W =\{x \in \mc{X} : d_V(x,\rho_V^W) \leq E \text{ for all } V \trans W \text{ or } W \propnest V\}.\]
\end{defn}

The main properties of $\P_W$ that we shall need are given in the following proposition.

\begin{prop}\label{prop:product_regions}
	Let $(\mc{X},\mf{S})$ be a hierarchically hyperbolic space with constant $E$.
	\begin{enumerate}
		\item \label{P_prop:HQC} There exists  $k \colon [0,\infty) \to [0,\infty)$ depending only on $E$ so that $\P_W$ is $k$--hierarchically quasiconvex for all $W \in \mf{S}$.
		\item \label{P_prop:large_projection} For all $W,V \in \mf{S}$, if $\diam(\pi_V(\P_W)) > 3E$, then $W \in \mf{S}_V \cup \mf{S}_V^\perp$.
		\item \label{P_prop:minimal_HQC} Suppose $\mc{Y} \subseteq \mc{X}$ is $k$--hierarchically quasiconvex  and $W \in \mf{S}$.
		For all $C \geq0$ there exists $\nu=\nu(C,E,k) \geq 0$ so that if $\pi_W
		\vert_{\mc{Y}}$ is $C$--coarsely onto for all $W\in \mf{S}_V \cup
		\mf{S}_V^{\perp}$, then $\P_V \subseteq \mc{N}_{\nu}(\mc{Y})$.
		\item \label{P_prop:wide} If $\mf{S}_W \cap \mf{S}^\infty$ and $\mf{S}_W^\perp \cap \mf{S}^\infty$ are both non-empty, then $\P_W$ is uniformly quasi-isometric to the direct product of two infinite diameter, quasi-geodesic metric spaces.
	\end{enumerate}
\end{prop}

While we will not need this structure directly, there are two additional hierarchically quasiconvex subsets, $\F_W$ and $\mathbf{E}_W$, so that $\P_W$ is naturally quasi-isometric to the product $\F_W \times \mathbf{E}_W$ (this is the quasi-isometry in Item \eqref{P_prop:wide}).

\subsection{The boundary of a hierarchically hyperbolic space}
Durham, Hagen, and Sisto defined a boundary for an HHS
$(\mc{X},\mf{S})$ that is built from the boundaries of the hyperbolic
spaces in $\mf{S}$; \cite{HHS_Boundary} is the reference for this 
subsection.  

	We first recall the construction of the boundary of a hyperbolic space. Let $X$ be a $\delta$--hyperbolic metric space. For any $x,y,z \in X$, the \emph{Gromov product of $x$ and $y$ with respect to $z$} is \[ \gprel{x}{y}{z} := \frac{1}{2} \left(d_X(x,z) + d_X(y,z) - d_X(x,y) \right).\]

Given a fixed basepoint $x_0$ of $X$, a sequence of points $(x_n)$ in $X$  \emph{converges to infinity} if $$\gp{x_n}{x_k}  \to  \infty$$ as
$n,k \to \infty$. Two sequences $(x_n)$ and $(y_n)$ are \emph{asymptotic} if $\gp{x_n}{y_n} \to \infty$ as $n \to \infty$. Note, this is equivalent to requiring that $\gp{x_n}{y_k} \to \infty$ as $n,k \to \infty$.  The \emph{Gromov boundary} $\partial X$ of  $X$ is the set of sequences in $X$ that converge to infinity modulo the equivalence relation of being asymptotic. 

The Gromov product extends to $x,y \in X \cup \partial X$ and $z \in X$ by taking the supremum of $$\liminf\limits_{n,k}\gprel{x_n}{y_k}{z}$$ over all sequences $(x_n)$ and $(y_k)$ that are either asymptotic to $x$ or $y$ when they are boundary points or converge to $x$ or $y$ when they are points in $X$. We topologize $X \cup \partial X$ by declaring a sequence $(x_n)$ in $X \cup \partial X$ to converge to $x \in \partial X$ if and only if $$\lim_{n \to \infty} \gp{x_n}{x} = \infty.$$	

\begin{defn}\label{def:nbhdbasis}
	For each $p \in \partial X$, the sets 
	\[
	M(r;p) = \left\{ x \in X \cup \partial X : \gp{p}{x} > r\right\} \] where
	$r >0$ form a neighborhood basis for $p$ in $X \cup \partial X$.  Note that
	if $r \leq r'$, then $M(r';p) \subseteq M(r;p)$.
\end{defn}

Despite the presence of the basepoint in the above definitions: convergence to infinity, being asymptotic, the Gromov boundary, and the topology of $X \cup \partial X$ are all independent of the choice of basepoint.

We now describe the boundary of a hierarchically hyperbolic space. The points in the HHS boundary are organized in a simplicial
complex that we denote $\partial_\Delta(\mc{X},\mf{S})$.  The
vertex set of $\partial_\Delta(\mc{X},\mf{S})$ is the set of all
boundary points of all the hyperbolic spaces $\mc{C}W$ for $W \in
\mf{S}^{\infty}$.  That is, the set of vertices is the set of points 
$\bigcup_{W\in \mf{S}^{\infty}} \partial \mc{C}W$.  The vertices
$p_1,\dots,p_n$ of $\partial_\Delta(\mc{X},\mf{S})$ will form an
$n$--simplex if each $p_i \in \partial \mc{C}W_i$ and $W_i \perp W_j$
for each $i \neq j$.  This means the set of points making up the HHS
boundary can equivalently be described as the set of all linear
combinations $\sum_{W\in\mf{W}} a_W p_W$ where
\begin{itemize}
	\item $\mf{W}$ is a pairwise orthogonal subset of $\mf{S}^{\infty}$,
	\item$p_W \in \partial \mc{C}W$  for each $W \in\mf{W}$,  and
	\item $\sum_{W\in\mf{W}} a_W = 1$ and each $a_W > 0$.
\end{itemize}

\begin{defn}\label{def:suppset}
	For each $p \in \partial_\Delta (\mc{X},\mf{S})$, we define
	$\supp(p)$, the \emph{support of $p$}, to be the pairwise orthogonal
	set $\mf{W} \subseteq \mf{S}$ so that $p = \sum_{W\in\mf{W}} a_W p_W$.
	Equivalently, the support of $p$ is the pairwise orthogonal set
	$\mf{W} \subseteq \mf{S}$ so that the smallest dimensional simplex of
	$\partial_\Delta (\mc{X},\mf{S})$ that contains $p$ has exactly one
	vertex from $\partial \mc{C}W$ for each $W \in\mf{W}$.
\end{defn}

Durham, Hagen, and Sisto equip the HHS boundary with a topology beyond that coming from the simplicial complex described above. We use $\partial (\mc{X},\mf{S})$ to denote the HHS boundary equipped with this topology, while $\partial_\Delta(\mc{X},\mf{S})$ will denote the simplicial complex that is the underlying set of boundary points.

  The definition of the topology on $\partial (\mc{X},\mf{S})$  is quite involved, combining the standard topology on the boundaries of the hyperbolic spaces $\mc{C}W$ with projections of boundary points onto certain domains of the HHS structure. When $\mc{X}$ happens to be hyperbolic, this topology is naturally homeomorphic to the Gromov boundary $\partial \mc{X}$.  As we will not need the full definition of the boundary, we will  cite the relevant properties as we need them and direct the curious reader to \cite{HHS_Boundary} for the definition of the topology.
  
  The topology on $\partial (\mc{X},\mf{S})$ can be extended to a topology on $\mc{X} \cup \partial(\mc{X},\mf{S})$ so that sequences in $\mc{X}$ can converge to points in $\partial(\mc{X},\mf{S})$. This allows us to define the limit set of a subset of $\mc{X}$. 
    
  \begin{defn}
	Let $(\mc{X},\mf{S})$ be an HHS and $\mc{Y} \subseteq \mc{X}$. Define the \emph{limit set} of  $\mc{Y}$ in $\partial (\mc{X},\mf{S})$ to be \[ \Lambda(\mc{Y}) := \{ p \in \partial(\mc{X},\mf{S}) : \text{there is a sequence } (y_n) \subseteq \mc{Y} \text{ converging to } p\}.\]
\end{defn}
   
  As with the topology on the boundary, we will forgo a complete description of the topology  on $\mc{X} \cup \partial (\mc{X},\mf{S})$ in favor of citing specific properties that we will need. For example, one immediate consequences of the definition of the topology is that sequences that converge to boundary points in $\mc{X}$ will project to sequences that converge to boundary points in the hyperbolic spaces $\mc{C}W$:
  
  \begin{lem}\label{lem:convergence_of_interior_points}
  	Let $(\mc{X},\mf{S})$ be an HHS. If $(x_n)$ is a sequence of points in $\mc{X}$ that converges to a point $p = \sum a_W p_W \in \partial (\mc{X},\mf{S})$, then for each $W \in \supp(p)$ and $x_n' \in \pi_W(x_n)$, the sequence $x_n'$ converges to $p_W$ in $\mc{C} W \cup \partial \mc{C}W$. 
  \end{lem}

   Just as in the Gromov boundary, pairs of sequence  in $\mc{X}$  at uniformly bounded distance  will converge to the same point in the boundary.
  
  \begin{lem}[{\cite[Lemma 3.20]{ABR:maximization}}]\label{lem:bounded_difference_convergence}
  	Let $(\mc{X},\mf{S})$ be an HHS. Let $(x_n)$ be a sequence of points in $\mc{X}$ that converges to  $p \in \partial (\mc{X},\mf{S})$. If $(y_n)$ is a  sequence in $\mc{X}$ with $ d_\mc{X}(x_n,y_n)$ uniformly bounded for all $n \in\mathbb{N}$, then $y_n$ also converges to $p$.
  \end{lem}

 When $\mc{X}$ is proper, the space $\mc{X}\cup \partial (\mc{X},\mf{S})$ is compact and Hausdorff \cite[Proposition 2.17 and Theorem 3.4]{HHS_Boundary}.  When $\mf{S}$ is a $G$--HHS structure, the action of $G$ on $(\mc{X},\mf{S})$ extends continuously to an action on $\partial(\mc{X},\mf{S})$ by homeomorphisms and simplicial automorphisms \cite[Corollary 6.1]{HHS_Boundary}.

\subsection{Maximization of HHS structures}\label{subsec:maximiztion}

The authors of \cite{ABD} described a process that takes an HHS structure $\mf{S}$ and produces a new HHS structure $\mf{T}$ with the following desirable properties. 

\begin{thm}[{\cite[Theorem 3.7]{ABD}}]\label{thm:maximization}
Let $(\mc{X},\mf{S})$ be an HHS with the unbounded domain dichotomy. There exists another HHS structure $\mf{T}$ for $\mc{X}$ so that
\begin{enumerate}
	\item \label{max:unbounded_domains} $\mf{T}$ has the unbounded domain dichotomy.
	\item \label{max:unbounded_products} For all $W \in \mf{T}$, both $\mf{T}_W \cap \mf{T}^\infty$ and $\mf{T}_W^\perp \cap \mf{T}^\infty$ are non-empty.
	\item \label{max:wide} For all $W \in \mf{T}$, the standard product region $\P_W$ is quasi-isometric to the product of two infinite diameter, quasi-geodesic spaces.
	\item\label{max:electrification} If $T \in \mf{T}$ is the $\nest$--maximal domain, then $\mc{C}T$ is the space obtained from $\mc{X}$ by adding edges $e_{xy}$ of length 1 between every pair of points $x,y$ with $x,y \in \P_W$ for some $W \in \mf{T}-\{T\}$.
\end{enumerate}
Moreover, if $\mf{S}$ is a $G$--HHS structure for some finitely generated group $G$, then $\mf{T}$ will also be a $G$--HHS structure.
\end{thm}
 
We  call the structure $\mf{T}$ produced from $\mf{S}$ in Theorem \ref{thm:maximization} the \emph{maximization of $\mf{S}$}. We will say that an HHS structure on $\mc{X}$ is \emph{maximized} if it is obtained by applying Theorem \ref{thm:maximization} to some HHS structure.

In \cite{ABR:maximization}, we showed that the maximization process in Theorem \ref{thm:maximization} does not change the HHS boundary nor which subsets are  hierarchically quasiconvex.

\begin{thm}[{\cite[Theorem 4.1 and Proposition 4.9]{ABR:maximization}}]\label{thm:maximization_invariance}
	Let $(\mc{X},\mf{S})$ be an HHS with the unbounded domain dichotomy, and let $\mf{T}$ be the maximization of $\mf{S}$.
	\begin{enumerate}
		\item\label{invar:boundary} If $\mc{X}$ is proper, then the identity map $\mc{X} \to \mc{X}$ continuously extends to a map $\partial(\mc{X},\mf{S}) \to \partial(\mc{X},\mf{T})$ that is both a homeomorphism and a simplicial automorphism.
			\item \label{invar:HQC} A subset $\mc{Y} \subseteq \mc{X}$ is hierarchically quasiconvex with respect to $\mf{S}$ if and only if it is hierarchically quasiconvex with respect to $\mf{T}$. Moreover, the function of
		hierarchical quasiconvexity in either $\mf{S}$ or $\mf{T}$ will
		determine the function in the other.
	\end{enumerate}

\end{thm}

\noindent In light of Theorem \ref{thm:maximization_invariance}, we will frequently assume that the HHS structures we are working with are maximized. When working with maximized structures, we will commonly make use of the  properties in Theorem \ref{thm:maximization}, particularly Item \eqref{max:unbounded_products}, without comment.

\subsection{Relative hyperbolicity}

Several equivalent formulations of (strong) relatively hyperbolicity exist in the literature. We will work with one in terms of the addition of combinatorial horoballs. The equivalence of this definition with other common definitions is shown in \cite{SistoMetRel}.

We first establish our model for horoballs.

\begin{defn}
	Let $\Gamma$ be a connected graph with vertex set $V$ and edge set $E$. Suppose each edge of $\Gamma$ has length 1. The \emph{combinatorial horoball} $\mc{H}(\Gamma)$ is the graph with vertex set $V \times \mathbb{Z}_{\geq 0}$ and two types of edges:
	\begin{itemize}
		\item for each  $n \in \mathbb{Z}_{\geq 0}$ and $v \in V$, there is an edge of length 1 between $(v,n)$ and $(v,n+1)$;
		\item for each  $n \in \mathbb{Z}_{\geq 0}$ and  $v,w \in V$ with $(v,w) \in E$, there is an edge of length $e^{-n}$ between $(v,n)$ and $(w,n)$.
	\end{itemize}
\end{defn}

The combinatorial horoball $\mc{H}(\Gamma)$ is always a hyperbolic space with a single boundary point. The constant of hyperbolicity is independent of $\Gamma$.

Since our horoballs are only defined for graphs, we use the following approximation graphs to construct horoballs for arbitrary subsets.
\begin{defn}
	A subset $P$ of a geodesic metric space $X$ is 
	\emph{$C$--coarsely connected} if every pair of points in $P$ can 
	be joined by a path that is contained in $\mc{N}_C(P)$. For a  $C$--coarsely connected subset $P$, a \emph{$C$--net} $N$ in $P$ is a subset of points of $P$ so that every point of $P$ is within $2C$ of a point in $N$ and every pair of points in $N$ are at least $C$ apart. An \emph{approximation graph} for $P$ is the graph whose vertex set is a $C$--net in $P$ with an edge of length 1 between two points if they are $2C$ apart.
\end{defn}

Finally, we define a relatively hyperbolic space as one that produces a hyperbolic space after attaching a collection of horoballs to subsets.

\begin{defn}
	Let $X$ be a geodesic metric space and $\mc{P}$ a collection of $C$--coarsely connected subsets of $X$. For each $P\in \mc{P}$, let $N_P$ be a $C$--net for $P$, and let  $\Gamma_P$ be the approximation graph for $P$ whose vertex set is $N_P$. A \emph{cusped spaced} for $X$ relative to $\mc{P}$ is the space obtained from $X \sqcup \bigsqcup_{P \in \mc{P}} \mc{H}(\Gamma_P)$ by adding an edge of length one between each point $v \in N_P$ and the vertex $(v,0) \in \mc{H}(\Gamma_P)$. We say $X$ is \emph{hyperbolic relative to $\mc{P}$} if some (hence any) cusped space for $X$ relative to $\mc{P}$ is Gromov hyperbolic. 
	
	We use $\cusp(X,\mc{P})$ to denote the cusped space of $X$ relative to $\mc{P}$. Up to quasi-isometry, this space does not depend on the choice of approximation graph for elements of $\mc{P}$. When $X$ is hyperbolic relative to $\mc{P}$, we use $\mc{H}(P)$ to denote the union of the horoball $\mc{H}(\Gamma_P)$, the subset $P$, and the edges between them. As with the cusp space, up to quasi-isometry, the horoballs are independent of the choice of approximation graph for $P$. The subsets of $\mc{P}$ are called the \emph{peripheral subsets} of $X$.
\end{defn}

In the case of finitely generated groups, we will require that 
the peripheral subsets of a relatively hyperbolic group are the cosets
of a collection of subgroups.  While a priori this appears to be a strong
condition, Dru\c{t}u showed in \cite[Theorem 1.5]{Drutu_Rel_Hyp_is_Geometric} 
that every finitely generated group which is a
relatively hyperbolic space is in fact hyperbolic relative to the
cosets of a finite collection of subgroups as described in the next
definition.

\begin{defn}
	A group $G$ is \emph{hyperbolic relative to subgroups 
	$H_1,\dots,H_k$} if some (hence any) Cayley graph of $G$ with
	respect to a finite generating set is hyperbolic relative to the
	collection of coset of $H_1,\dots,H_k$.  The subgroups
	$H_1,\dots,H_k$ are the \emph{peripheral subgroups} of $G$.  In this
	case, we use $\cusp(G, \{H_1,\dots,H_k\})$ to denote the space
	obtained by attaching  combinatorial horoballs to each coset of a
	peripheral subgroup in the Cayley graph of $G$.
\end{defn}

The basic idea of relative hyperbolicity is that all of the non-negative curvature must lie inside the individual peripheral subsets. This next result makes that explicit for subsets that are quasi-isometric to products.

\begin{thm}[{\cite[Corollary 5.8]{Drutu_Sapir_Rel_hyp}}]\label{thm:thick_subsets_of_rel_hyp}
	Let $X$ be a geodesic metric space that is hyperbolic relative to
	a collection of subsets $\mc{P}$.  If $Y$ is a subset of $X$ so that $Y$, equipped with the subset metric, is quasi-isometric to a product of two infinite diameter metric spaces, then $Y$ is contained in the $C$--neighborhood of some $P \in
	\mc{P}$,  where $C$ depends only on $X$, $\mathcal{P}$, and the quasi-isometry constants.
\end{thm}

For hierarchically hyperbolic spaces, the following criterion can be 
used to verify relative
hyperbolicity.

\begin{defn}\label{defn:isolated_orthogonality}
	Let $(\mc{X},\mf{S})$ be an HHS with the unbounded domain
	dichotomy.  We say $\mf{S}$ has \emph{orthogonality isolated by $\mf{I}
	\subseteq \mf{S}$} if
	\begin{enumerate}
		\item $\mf{I}$ does not contain the $\nest$--maximal element of $\mf{S}$;
		\item if $V,W \in \mf{S}$ and $V \perp W$, then there exists $I \in \mf{I}$ so that $V,W \propnest I$; and
		\item if $W \in \mf{S}$ and there exist $I_1,I_2 \in \mf{I}$ so that $W \nest I_1,I_2$, then $I_1 = I_2$.
	\end{enumerate}
\end{defn}

\begin{thm}[{\cite[Theorem 4.2]{Russell_Relative_Hyperbolicity}}]\label{thm:isolated_orthogonality_implies_rel_hyp}
	Let $(\mc{X},\mf{S})$ be an HHS with the bounded domain dichotomy. If $\mf{S}$ has orthogonality isolated by $\mf{I} \subseteq \mf{S}$, then $\mc{X}$ is hyperbolic relative to $\{\P_I : I \in \mf{I}\}$.
\end{thm}

When $G$ is a relatively hyperbolic $G$--HHS, not every $G$--HHS structure for
$G$ must have isolated orthogonality.   However Corollary
\ref{cor:rel_hyp_structure} will show that every relatively
hyperbolic $G$--HHS has at least one $G$--HHS structure with isolated
orthogonality. Russell originally established this result for 
hierarchically hyperbolic groups satisfying the additional hypothesis of \emph{clean containers}; see \cite[Section 5]{Russell_Relative_Hyperbolicity}.

\subsection{Hyperbolically embedded subgroups}	

A key feature of the peripheral subgroups of relatively hyperbolic groups is that they are \emph{hyperbolically embedded}. As we will not need the precise definition of a hyperbolically embedded subgroup,  we forgo it in favor of Theorem \ref{thm:hyperbolically_embedded_in_HHGs} below, which provides a characterization of hyperbolically embedded subgroups of $G$--HHSs.

\begin{defn}
	A collection of subgroups $\{H_1,\dots,H_k\}$ of a group $G$ is
	\emph{almost malnormal} if \[|gH_ig^{-1} \cap H_j| = \infty
	\implies i =j \text{ and } g \in H_i.\] 
\end{defn}

\begin{defn}
	A subset $Y$ of a metric
	space $X$ is \emph{$M$--strongly quasiconvex} if there exists a
	function $M \colon [1,\infty) \times [0,\infty) \to [0,\infty)$ so
	that every $(\lambda,\varepsilon)$--quasi-geodesic with endpoints in
	$Y$ is contained in the $M(\lambda,\varepsilon)$--neighborhood of
	$Y$.  A subgroup $H$ of a finitely generated group $G$ is
	\emph{strongly quasiconvex} if $H$ is a strongly quasiconvex
	subset of the Cayley graph of $G$ with respect to a finite
	generating set.
\end{defn}

\begin{thm}[{\cite{DGO_rotating_families,Sisto_quasiconveity_of_hyperbolically_embedded}, \cite[Theorem
8.1]{RST_Quasiconvexity}}]\label{thm:hyperbolically_embedded_in_HHGs}
Let $\{H_1,\dots,H_k\}$ be a collection of subgroups of a finitely generated group $G$.  If  $\{H_1,\dots,H_n\}$ 
is hyperbolically embedded, then it is an almost malnormal and each $H_i$ is strongly quasiconvex.
	Moreover, the converse holds when $G$ is a $G$--HHS.\footnote{In
	\cite{RST_Quasiconvexity}, this result is stated for HHGs, but the
	proof goes through as is for $G$--HHSs.} 
\end{thm}

The next definition and result describe how  strong quasiconvexity can be detected using the hierarchically hyperbolic structure.

\begin{defn}\label{defn:ortho_proj_dich}
	A subset $\mc{Y}$ of a hierarchically hyperbolic space
	$(\mc{X},\mf{S})$ has the \emph{$B$--orthogonal projection dichotomy} if
	whenever there exists $W \in \mf{S}$ satisfying $\diam(\pi_W(\mc{Y})) > B$, the projection
	$\pi_V\vert_{\mc{Y}}$ is $B$--coarsely onto for all $V \in
	\mf{S}_W^\perp$.
\end{defn}

\begin{thm}[{\cite[Theorem 6.2]{RST_Quasiconvexity}}]\label{thm:SQC_in_HHG}
	Let $(\mc{X},\mf{S})$ be an HHS with the bounded domain dichotomy.
	\begin{enumerate}
		\item Given $k \colon [0,\infty) \to [0,\infty)$ and $B \geq 0$, there exists $M \colon [1,\infty) \times [0,\infty) \to [0,\infty)$, so that if $\mc{Y} \subseteq \mc{X}$ is $k$--hierarchically quasiconvex and has the $B$--orthogonal projection dichotomy, then $\mc{Y}$ is $M$--strongly quasiconvex.
		\item Given $M \colon [1,\infty) \times [0,\infty) \to [0,\infty)$, there exists $k \colon [0,\infty) \to [0,\infty)$ and $B \geq 0$, so that if $\mc{Y} \subseteq \mc{X}$ is $M$--strongly quasiconvex, then $\mc{Y}$ is $k$--hierarchically quasiconvex and has the $B$--orthogonal projection dichotomy.
	\end{enumerate}
\end{thm}

Lastly, we record a simple but handy fact about the intersection of
cosets of almost malnormal collections of subgroups.  Since every
hyperbolically embedded collection of subgroups is  almost
malnormal, this lemma applies to any hyperbolically embedded
collection, which is how we will apply it.

\begin{lem}[{\cite[Proposition 9.4]{Hruska_relative_quasiconvexity}}]\label{lem:intersection_of_almost_malnormal}
	Let $G$ be a finitely generated group and $\{H_1,\dots, H_k\}$ be an almost malnormal collection of subgroups. For each $C \geq 0$ and any two cosets $gH_i$ and $hH_j$, we have $$ \diam \left( \mc{N}_C(gH_i) \cap \mc{N}_C(hH_j)\right) = \infty \implies gH_i = hH_j.$$
\end{lem}

\section{Adding hyperbolically embedded subgroups to a structure}\label{sec:add_hyp_embedded}
\newtheorem{con}[thm]{Construction}

In this section, we show that any collection of hyperbolically embedded
subgroups of a maximized $G$--HHS can be naturally associated to a set of domains 
in an $G$--HHS structure on the group. 
We begin by describing the structure.

\begin{con}\label{con:HHG+hyp_embedded}
	Let $\mf{S}$ be a maximized $G$--HHS structure for the finitely generated group $G$.  Let $S \in \mf{S}$ be the
	$\nest$--maximal element of $\mf{S}$.  Let $\{H_1,\dots, H_k\}$ be
	a collection of hyperbolically embedded subgroups of $G$.  Let
	$\mf{Q}$ be a set indexing the set of cosets of $H_1,\dots, H_k$.
	For each $Q \in \mf{Q}$, we will use $\coset{Q}$ to denote the
	coset in $G$ that is indexed by $Q$.  We  describe a new $G$--HHS
	structure for $G$ whose index set includes $\mf Q$.
	\begin{itemize}
		\item Index set: $ \mf{H} = \mf{S} \cup\mf{Q}$.
		\item Hyperbolic spaces: For $S$, the space $\mc{C}_\mf{H}S$ is  obtained from $\mc{C}_\mf{S}S$ by adding an edge between every pair of points in   $\pi_S(\coset{Q})$ for each $Q \in \mf{Q}$.   Following \cite{Farb_Rel_Hyp}, we call this the \textit{electrified space}.  For $V \in \mf{S}-\{S\}$, define $\mc{C}_{\mf{H}}V := \mc{C}_\mf{S} V$. For $Q \in \mf{Q}$, let $\mc{C}_\mf{H}Q$ be the convex hull of $\pi_S(\coset{Q})$ in the space $\mc{C}_\mf{S}S$. 
		
		\item Projection maps: We use $\tau_\ast$ to denote the projection maps in $\mf{H}$ and $\pi_\ast$ to denote the projection maps in $\mf{S}$. For $V \in \mf{S}-\{S\}$, let $\tau_V := \pi_V$. For $S$, the map $\tau_S$ is the composition of $\pi_S$ with the inclusion $\mc{C}_\mf{S}S \to \mc{C}_\mf{H}S$.  For $Q \in \mf{Q}$, the map $\tau_{Q}$ is the composition $\mf{p}_{\pi_S(\coset{Q})} \circ \pi_S$. 
		
		\item Relations: For all $V,W \in \mf{S}$, the relation in $\mf{H}$ between $V$ and $W$ is the same as the relation between $V$ and $W$ in $\mf{S}$. Each $Q \in \mf{Q}$ is properly $\mf{H}$--nested into $S$. For $V \in \mf{S} -\{S\}$ and $Q \in \mf{Q}$, we define $V \propnest Q$ if there exist $W \in \mf{S}^\infty \cap \mf{S}_V^\perp$ so that $\pi_W\vert_{\coset{Q}}$ is coarsely  onto; otherwise $Q \trans V$. If $Q,R \in\mf{Q}$ are not equal, then $ Q \trans R$.
		
		\item Relative projections: We use $\beta_\ast^\ast$ to denote the relative projections in $\mf{H}$ and $\rho_\ast^\ast$ to denote them in $\mf{S}$. For all $V,W \in \mf{S}$, if $V \propnest W$ or $V \trans W$, then $\beta_W^V := \rho_W^V$.  For $Q \in \mf{Q}$, the relative projection $\beta_S^{Q}$ is the electrified subset $\tau_S(\coset{Q})$ in $\mc{C}_\mf{H}S$. For $V \in \mf{S}$ and $Q \in \mf{Q}$, if $V \propnest Q$ or $V \trans Q$, then the relative projection $\beta_{Q}^V$ is $\mf{p}_{\pi_S(\coset{Q})}(\rho_S^V)$. If $Q \trans W$ for any $W \in \mf{H}$, then $\beta_W^{Q} := \tau_W(\coset{Q})$.
	\end{itemize}
\end{con}
	
	While the reader should think of the set $\mf{Q}$ as the set of
all coset of $H_1, \dots, H_k$, we note again that formally,  
the element $Q \in \mf{Q}$ is
an element of the index set $\mf{Q} \subset \mf{H}$ while
$\coset{Q}$ refers to the actual coset of a $H_i$ in the group
$G$.  We choose this notation because the coset $\coset{Q}$ 
coarsely coincide with the product region $\mathbf{P}_Q$ in $\mf{H}$ as follows.

\begin{rem}[Product regions for $\mf{H}$]\label{rem:products_in_h} 
	For each non-$\nest$--maximal $V \in \mf{S}$, the set $\mf{S}_V$ (resp. $\mf{S}_V^{\perp}$) and the corresponding collection of hyperbolic spaces and projection maps  is identical to the set $\mf{H}_V$ (resp. $\mf{H}_V^{\perp}$) and its corresponding collection of hyperbolic spaces and projection maps. Hence, the product regions for $V$ with respect to both $\mf{H}$ and $\mf{S}$ are identical. For $Q \in \mf{Q}$, the product region $\mathbf P_Q$ with respect to $\mf{H}$ is finite Hausdorff distance from the coset $\coset{Q}$, because 
	\begin{itemize}
		\item $\coset{Q}$ is uniformly hierarchically quasiconvex with respect to $\mf{H}$ (Corollary \ref{cor:properties_of_hyp_embed_cosets});
		\item $\mf{H}_Q^\perp =\emptyset$ by construction;
		\item  the projection of $\coset{Q}$ to every domain of $\mf{H}_Q$ is uniformly coarsely onto (Lemma \ref{lem:h_nesting_and_product_regions}); and
		\item the projection of $\coset{Q}$ to every domain of $\mf{H} - \mf{H}_Q$ is uniformly bounded (shown in the proof of Theorem \ref{thm:adding_in_hyperbolically_embedded_subgroups}) .
	\end{itemize}
\end{rem}

We now collect some results we will need to show that the structure $\mf{H}$ is in fact a $G$--HHS structure. We will frequently use the following properties of the cosets of the hyperbolically embedded subgroups. The first is a  direct consequence of Theorems \ref{thm:hyperbolically_embedded_in_HHGs} and \ref{thm:SQC_in_HHG}, while the second was   shown during the proof of Theorem \ref{thm:hyperbolically_embedded_in_HHGs}; see  \cite[Proposition 8.6]{RST_Quasiconvexity}.  

\begin{cor}\label{cor:properties_of_hyp_embed_cosets}
	Let $(G,\mf{S})$ be a $G$--HHS and $\{H_1,\dots, H_k\}$ a
	hyperbolically embedded collection of subgroups.  Let $\mf{Q}$ be
	the set indexing the cosets of the $H_i$ as in Construction
	\ref{con:HHG+hyp_embedded}.  There exists $k \colon [0,\infty) \to
	[0,\infty)$ and $B \geq 0$ so that
	\begin{enumerate}
		\item for each $Q \in \mf{Q}$, the coset $\coset{Q}$ is $k$--hierarchically quasiconvex  and has the $B$--orthogonal projection dichotomy; and
		\item for distinct $Q,R \in \mf{Q}$, the diameter of $\mf{p}_{\pi_S(F(Q))}(\pi_S(\coset{R}))$ is at most $B$.
	\end{enumerate}
\end{cor}

The nesting relation in $\mf{H}$ is defined in order to facilitate the following key lemma.

\begin{lem}\label{lem:h_nesting_and_product_regions}
	Let $(G,\mf{S})$ be a maximized $G$--HHS. Suppose $\{H_1,\dots, H_k\}$ is a hyperbolically embedded collection of subgroups of $G$, and let $\mf{H}$ be the structure described in Construction \ref{con:HHG+hyp_embedded}. There exists $B \geq0$ so that for each $V \in \mf{S}$ and $Q \in \mf{Q}$, the following are equivalent.
	\begin{enumerate}
		\item $V \nest Q$ in $\mf{H}$. \label{item:h_nesting}
		\item There is $W \in \mf{S}_V^{\perp}$ with $\diam(\pi_W(\coset{Q})) > B$. \label{item:uniform_h_nesting}
		\item There is $U \in \mf{S}_V \cap \mf{S}^\infty$ with $\pi_U \vert_{\coset{Q}}$ coarsely onto. \label{item:onto_nested}
		\item There is $U \in \mf{S}_V$ with $\diam(\pi_U(\coset{Q}) > B$. \label{item:uniform_onto-nested}
		\item The product region $\mathbf{P}_V$ is contained in the 
		$B$--neighborhood of the coset $\coset{Q}$. \label{item:uniform_P_containment}
		\item The product region $\mathbf{P}_V$ is contained in a finite neighborhood of the coset $\coset{Q}$. \label{item:P_containment}
	\end{enumerate}
\end{lem}

\begin{proof} By Corollary \ref{cor:properties_of_hyp_embed_cosets}, for each $Q \in \mf{Q}$ the coset $\coset{Q}$ has the $B_0$--orthogonal
	projection dichotomy for some $B_0$ determined by $(G,\mf{S})$
	and $\{H_1, \dots, H_k\}$. Moreover, we can assume $B_0$ is large enough that for all $V \in \mf{S}$,  if $\diam(\mc{C}_\mf{S}V) >B_0 $, then $ \diam(\mc{C}_\mf{S}V) = \infty$.
	
	We will first prove that (\ref{item:h_nesting}) implies  (\ref{item:uniform_h_nesting}) through (\ref{item:uniform_onto-nested}) for any $B > B_0$.
	\begin{claim}
		Item (\ref{item:h_nesting}) $\implies$ Item (\ref{item:uniform_h_nesting}) $\implies$ Item  (\ref{item:onto_nested}) $\implies$ Item  (\ref{item:uniform_onto-nested}).
	\end{claim}
	
	\begin{proof}
		A domain $V\in\mf{S}$ nests into $Q \in \mf{Q}$ in the structure $\mf{H}$ if and only if there exists $W \in
		\mf{S}_V^{\perp} \cap \mf{S}^{\infty}$ so that
		$\pi_W\vert_{\coset{Q}}$ is coarsely onto.  Since
		$\diam(\mc{C}_\mf{S} W) =\infty$, Item
		(\ref{item:uniform_h_nesting}) holds.
		
	  Now, if $\diam(\pi_W(\coset{Q})) \geq B_0$ for
		some $W \in \mf{S}_V^{\perp} \cap \mf{S}^\infty$, then
		$\pi_{U} \vert_{\coset{Q}}$ is $B_0$--coarsely onto for any 
		domain $U$ orthogonal to $W$, and, in particular, for all  $U
		\in \mf{S}_V$.  Since $\mf{S}$ is maximized, we know $\mf{S}_V
		\cap \mf{S}^{\infty} \neq \emptyset$.  Thus  Item
		(\ref{item:onto_nested}) 
		follows from (\ref{item:uniform_h_nesting}).  Item  (\ref{item:uniform_onto-nested}) follows immediately from Item (\ref{item:onto_nested}), as $U\in \frak S^{\infty}$.	\end{proof}
	
	Next we show that Item (\ref{item:uniform_onto-nested}) implies that $\mathbf{P}_V$ is contained in  the $B_1$--neighborhood of $\coset{Q}$ for some $B_1$ determined by $B_0$ and the hierarchy constant for $\mf{S}$.
	
	\begin{claim}
		Item \eqref{item:uniform_onto-nested} $\implies$ Item \eqref{item:uniform_P_containment}.
	\end{claim}

	\begin{proof}
		Let $V \in \mf{S}$ and $Q \in \mf{Q}$, and assume $\diam(\pi_U(\coset{Q})) > B_0$ for some $U \in \mf{S}_V$. By Corollary \ref{cor:properties_of_hyp_embed_cosets}, $\coset{Q}$ is uniformly hierarchically quasiconvex. By Proposition \ref{prop:product_regions}\eqref{P_prop:minimal_HQC},  if we can show $\pi_W \vert_{\coset{Q}}$ is $B_0$--coarsely onto for each $W \in \mf{S}_V \cup \mf{S}_V^{\perp}$, then there will be a constant $B_1\geq 0$ depending on $B_0$ so that  $\P_V$ is  contained in the $B_1$--neighborhood of $\coset{Q}$.   
		
		First suppose that $W \in \mf{S}_V^\perp$. Since $U \nest V$, we have $U \perp W$. By the $B_0$--orthogonal projection dichotomy,  $\diam(\pi_U(\coset{Q})) > B_0$ implies $\pi_W \vert_{\coset{Q}}$ is $B_0$--coarsely onto. 
				
		Now consider $W \in \mf{S}_V$. Since $\mf{S}$ is maximized, there must exist $Z \in \mf{S}^\infty \cap \mf{S}_V^\perp$. As shown in the proceeding paragraph, $\pi_Z \vert_{\coset{Q}}$ is $B_0$--coarsely onto. However, since $\diam(\mc{C}_\mf{S}Z) =\infty$ and $W \perp Z$, the $B_0$--orthogonal projection dichotomy implies that $\pi_W \vert_{\coset{Q}}$ is $B_0$--coarsely onto, as well.
	\end{proof}

	Since Item (\ref{item:uniform_P_containment}) automatically implies Item (\ref{item:P_containment}), it remains to show  Item (\ref{item:P_containment}) implies  $V \nest Q$.
	
	\begin{claim}
		Item \eqref{item:P_containment} $\implies$ Item \eqref{item:h_nesting}.
	\end{claim}
	
	\begin{proof}
		Let $V \in \mf{S}$ and $Q \in \mf{Q}$.  Assume that
		$\mathbf{P}_V$ is contained in a regular neighborhood of
		$\coset{Q}$.  Since $\coset{Q}$ does not coarsely equal all of $G$, it must be the case that $V$ is not $\nest$--maximal. By Lemma \ref{lem:h_nesting_and_product_regions}, 
		the restriction of $\pi_W$ to $\P_V$ is coarsely onto for all $W \in \mf{S}_V \cup \mf{S}_V^{\perp}$.  In
		particular, $\pi_W \vert_{\coset{Q}}$ must also be coarsely
		onto, because $\pi_W$ is coarsely Lipschitz and $\P_V$ is
		contained in a regular neighborhood of $\coset{Q}$.  Because
		$\mf{S}$ is maximized, we know $\mf{S}^\infty \cap \mf{S}_V^\perp
		\neq \emptyset$.
		Hence $V \nest Q$ because there must exist $W
		\in \mf{S}^{\infty} \cap \mf{S}_V^{\perp}$ with $\pi_W
		\vert_{\coset{Q}}$ coarsely onto.
	\end{proof}

Lemma \ref{lem:h_nesting_and_product_regions} now holds with $B = \max\{B_0,B_1\}$.
\end{proof}

We now turn to the main result of this section, in which  we establish that the structure in Construction \ref{con:HHG+hyp_embedded} is a $G$--HHS structure. 

\begin{thm}\label{thm:adding_in_hyperbolically_embedded_subgroups}
	Let $(G,\mf{S})$ be a maximized $G$--HHS. Let $S \in \mf{S}$ be
	the $\nest$--maximal element of $\mf{S}$ and $\{H_1,\dots,H_k\}$
	be a hyperbolically embedded collection of subgroups of $G$.  The
	structure $\mf{H}$ described in Construction
	\ref{con:HHG+hyp_embedded} is a $G$--HHS structure.
	
	Moreover, if $(G,\mf{S})$ is a hierarchically hyperbolic group 
	for which $\mf{S}$ is a maximized structure, then 
	$\mf{H}$ is a hierarchically hyperbolic group structure for
	$G$. 
\end{thm}

\begin{rem}
	The moreover clause applies to a number of natural examples,  
	including 
	the standard HHG structures on RAAGs and on mapping class 
	groups, since these are maximized HHG structures.
\end{rem}

Before proving Theorem \ref{thm:adding_in_hyperbolically_embedded_subgroups} we record two short observations. First, adding the hyperbolically embedded subgroups to the structure does not change the HHS boundary. Second, when $G$ is hyperbolic relative to the $H_i$, the structure $\mf{H}$ has isolated orthogonality.

\begin{cor}\label{cor:adding_hyp_embedded+boundary}
	Let $(G,\mf{S})$ be a maximized $G$--HHS, then let $\mf{H}$ be the HHG structure from Construction \ref{con:HHG+hyp_embedded} for a collection of hyperbolically embedded subgroups $\{H_1,\dots, H_k\}$. There is a homeomorphism $\Phi \colon G \cup \partial (G,\mf{S}) \to G \cup \partial (G,\mf{H})$ so that $\Phi$ restricts to the identity on $G$ and to both a homeomorphism and a simplicial isomorphism  $\partial(G,\mf{S}) \to \partial(G,\mf{H})$.
\end{cor}

\begin{proof}
	Since $\mf{H}$ has the same orthogonality relations as $\mf{S}$, the maximization of $\mf{H}$ is identical to the maximization of $\mf{S}$; see \cite[Theorem 3.7]{ABD}. The corollary is therefore a consequence of Theorem \ref{thm:maximization_invariance}\eqref{invar:boundary}.
\end{proof}

\begin{cor}\label{cor:rel_hyp_structure}
	Let $(G,\mf{S})$ be a maximized $G$--HHS that is hyperbolic
	relative to a finite collection of subgroups $\{H_1,\dots, H_k\}$.
	Let $\mf{H}$ be the $G$--HHS structure of Construction
	\ref{con:HHG+hyp_embedded} obtained by adding the cosets of the
	subgroups $\{H_1,\dots, H_k\}$.  The structure $\mf{H}$ has
	orthogonality isolated by $\mf{Q}$, and every non-$\nest$--maximal
	domain in $\mf H$ nests into some $Q\in\mf Q$.
\end{cor}

\begin{proof}
	Let $W,V \in \mf{H}$ with $W \perp V$. Since no two elements of $\mf{Q}$ are orthogonal, $W$ and $V$ must both be in $\mf{S}$. Since $\mf{S}$ is maximized, each $\P_W$ is uniformly quasi-isometric to the product of two infinite diameter quasi-geodesic spaces (Theorem \ref{thm:maximization}\eqref{max:wide}). Hence, Theorem \ref{thm:thick_subsets_of_rel_hyp} says each $\P_W$ must then be contained in a regular neighborhood of a coset $\coset{Q}$ for some $Q \in \mf{Q}$. Thus, $\pi_{U} \vert_{\coset{Q}}$ is coarsely onto for all $U \in \mf{S}_W$. Since $\mf{S}_W \cap \mf{S}^{\infty} \neq \emptyset$ and every element of $\mf{S}_W$ is orthogonal to $V$,  this implies that $W,V \propnest Q$ by Lemma \ref{lem:h_nesting_and_product_regions}.
	
	Now suppose $W \in \mf{H}$ is nested into both $Q, R \in \mf{Q}$.  Since all elements of $\mf{Q}$ are transverse, $W$ must be in $\mf{S}$. By Lemma \ref{lem:h_nesting_and_product_regions}, this implies $\P_W$ is contained in a regular neighborhood of both $\coset{Q}$ and $\coset{R}$. Because $\diam(\P_W) = \infty$, Lemma \ref{lem:intersection_of_almost_malnormal} says $\coset{Q} = \coset{R}$. Hence $Q = R$.
	
	For the last clause, note that because $\mf{S}$ is maximized, every non-$\nest$--maximal element of $\mf{S}$ is orthogonal to some domain of $\mf{S}$. Thus, the first paragraph above shows that every non-$\nest$--maximal $W \in \mf{S}$ nests into some $Q \in \mf{Q}$. 
\end{proof}

We now prove Theorem
\ref{thm:adding_in_hyperbolically_embedded_subgroups}.  A reader 
focused on the applications to the boundary, may skip it without a
loss of continuity for the remainder of the paper.

\begin{proof}[Proof of Theorem \ref{thm:adding_in_hyperbolically_embedded_subgroups}]
	The desired equivariance and finite orbit properties in Definition \ref{defn:nearlyHHG} of a $G$--HHS are satisfied for $\mf{H}$ by a combination of the fact that $\mf{S}$ is a $G$--HHS, the closet point projection in a hyperbolic spaces is coarsely equivariant under isometries, and that $\mf{Q}$ indexes a collection of cosets of a finite number of subgroups. Thus, it  suffices to prove that $\mf{H}$ is an HHS structure for $G$.

	We start by observing that $\mc{C}_{\mf S} S$ can be equipped with an HHS structure using the subsets $\pi_S(\coset{Q})$.  For each $Q \in \mf{Q}$, the set $\pi_S(\coset{Q})$ is uniformly quasiconvex in $\mc{C}_{\mf S} S$ because each $\coset{Q}$ is uniformly hierarchically quasiconvex in $(G,\mf{S})$. Further, if $Q \neq R$, then the closest point projection of $\pi_S(\coset{Q})$ onto $\pi_S(\coset{R})$ is uniformly bounded in $\mc{C}_{\mf S} S$  by Corollary \ref{cor:properties_of_hyp_embed_cosets}. Hence, the collection $\{\pi_S(\coset{Q}) : Q\in\mathfrak Q\}$ forms  what Spriano calls a  \textit{factor system} of $\mc{C}_{\mf S} S$; see \cite[Section 3]{Spriano_hyperbolic_I}. In particular, Spriano proves that  $\mc{C}_\mf{S}S$ has a hierarchically hyperbolic structure with index set $\mf{F} = \{S\} \cup \mf{Q}$, where the hyperbolic spaces are either the electrified space $\mc{C}_\mf{H}S$ or  $\mc{C}_\mf{H}Q$, the convex hull of $\pi_S(\coset{Q})$.  Each  element of $\mf{Q}$ is nested into $S$ and every pair of elements of $\mf{Q}$ are transverse. The projections and relative projections are all given by either inclusion or closest point projection in $\mc{C}_{\mf{S}}S$. This proves that $\mc{C}_\mf{H}S$ and each $\mc{C}_\mf{H} Q$ are uniformly hyperbolic, and will be useful when verifying the remaining axioms for $\mf{H}$ to be an $G$--HHS structure for $G$.

	Since $\mf{H}$ inherits many of the spaces, projection, and relations from $\mf{S}$, we only need to verify the HHS axioms for the domains in $\{S\} \cup \mf{Q}$. Let $B$ be larger than the constant from the bounded domain dichotomy for $\mf{S}$ and the constants from Corollary \ref{cor:properties_of_hyp_embed_cosets} and Lemma \ref{lem:h_nesting_and_product_regions}. Let $E \geq 1$ be the maximum of the hierarchy constants from both $\mf{S}$ and $\mf{F}$.
	
	\textbf{Hyperbolic spaces and projections:}  The hyperbolicity of $\mc{C}_{\mf H} S$ and each $\mc{C}_{\mf H} Q$ are shown above, and  $\tau_Q$ is uniformly coarsely Lipschitz because the maps $\pi_S$ and $\mf{p}_{\pi_S(\coset{Q})}$ are.
	
	\textbf{Nesting and finite complexity:} We need to verify that $\nest$ is still a partial order. It suffices to check that $\nest$ is still transitive when $V \propnest W$ in $\mf{S}$ and $W \propnest Q$ in $\mf{H}$ for some $Q\in\mf Q$.  In this case, there exists $U\in\mf{S}^\perp_W \cap \mf{S}^{\infty}$ so that $\pi_{U}\vert_{\coset{Q}}$ is coarsely onto. Since $V \propnest W$, we have $V \perp U$ as well. Hence $V \nest Q$ as desired. The maximal length of a $\propnest$--chain in $\mf{H}$ is at most 1 longer than the maximal length of a $\propnest$--chain in $\mf{S}$.
	
	The new upward relative projection are all bounded diameter, as they are either electrified subsets or  the closest point projection of a bounded diameter subset of $\mc{C}_{\mf S}S$.  
	
	\textbf{Orthogonality and containers:} Since the orthogonality relations in $\mf{S}$ and $\mf{H}$ are identical, these axioms are inherited from $\mf{S}$.
	
	\textbf{Transversality:} We only need to verify that $\beta_{Q}^{R}$, $\beta_V^{Q}$, and $\beta_{Q}^V$ have uniformly bounded diameter whenever $Q,R \in \mf{Q}$ and $Q \trans R$ or  $Q \in\mf{Q}$, $V\in\mf{S}$, and $Q \trans V$. 
	\begin{itemize}
		\item 	Since $\diam(\rho_S^V) \leq E$, the coarse Lipschtizness of $\mf{p}_{\pi_S(\coset{Q})}$ ensures $\beta_{Q}^V = \mf{p}_{\pi_S(\coset{Q})}\left(\rho_S^V \right)$ is uniformly bounded. 
		\item 	For $\beta_V^{Q} = \tau_V(\coset{Q}) =\pi_V(\coset{Q})$, observe that because $B$ is larger than the constant from Lemma \ref{lem:h_nesting_and_product_regions},  $\diam(\pi_V(\coset{Q})) >B$ would imply  $V \propnest Q$. Hence $\diam(\beta_V^Q) = \diam(\pi_V(\coset{Q})) \leq B$ when $V \trans Q$.
		\item By Corollary \ref{cor:properties_of_hyp_embed_cosets}, $\diam(\beta^R_Q)= \diam(\mf{p}_{\pi_{S}(\coset{Q})}(\pi_S(\coset{R})) \leq B$.
	\end{itemize}

	\textbf{Uniqueness:}  Let $x,y \in G$,   and suppose there exists $D\geq 0$ so that $d_V(\tau_V(x),\tau_V(y))\leq D$ for each $V \in \mf{H}$. By the uniqueness axiom in $(\mc{C}_{\mf{S}}S,\mf{F})$, there exists a bound $D' = D'(D,\mf{F})$ on the $\mc{C}_{\mf{S}}S$--distance between $\pi_S(x)$ and $\pi_S(y)$. Since $\tau_V = \pi_V$ for all $V \in \mf{S}-\{S\}$, the uniqueness axiom  for $(G,\s)$ then implies there exists a $D''=D''(D,\mf{S})$  bounding the distance between $x$ and $y$ in  $G$.
	
	\textbf{Bounded Geodesic Image:} We only need to verify the axiom when one of the two domains involved is either $S$ or $Q \in \mf{Q}$. Let $x,y \in G$. 
	
	We first handle the case of $ Q \nest S$ for some $Q \in \mf{Q}$. Assume that $d_{Q}(\tau_{Q}(x),\tau_{Q}(y)) >E$. By the bounded geodesic image axiom in $(\mc{C}_{\mf S}S,\mf{F})$, the $\mc{C}_{\mf H} S$--geodesic from $\tau_S(x)$ to $\tau_S(y)$ passes $E$--close to the electrified subset $\tau_S(\coset{Q}) = \beta_S^{Q}$.
	
	Next we verify the axiom when $V \in \mf{S}$ and $V \propnest S$ in $\mf{H}$. Assume that $d_V(\tau_V(x),\tau_V(y)) >E$. The bounded geodesic image axiom in $(G,\mf{S})$ implies the $\mc{C}_{\mf S} S$--geodesic from $\pi_S(x)$ to $\pi_S(y)$ intersects the $E$--neighborhood of $\rho_S^V$. Since $\mc{C}_{\mf S} S$ is hyperbolic, every geodesic in $ \mc{C}_{\mf S} S$ is a uniform hierarchy path in $(\mc{C}_{\mf S} S, \mf{F})$; see \cite[Proposition 3.5]{Spriano_hyperbolic_II}. Thus this geodesic, when viewed as a path in $\mathcal C_{\mf H} S$,  is a uniform quality quasi-geodesic connecting $\tau_S(x)$ and $\tau_S(y)$, and it intersects the $E$--neighborhood of $\beta_S^V = \rho_S^V$ as the map $\mc C_\frak S S \to \mc C_\frak H S$ is 1--Lipschitz.   Again using that $\mc C_{\mf H} S$ is hyperbolic, this implies every $\mc{C}_{\mf H} S$--geodesic from $\tau_S(x)$ to $\tau_S(y)$ will intersect a uniform neighborhood of $\beta_S^V$. 
	
	The last case is when $V \propnest Q$ for some $V\in\mf{S}$ and $Q \in \mf{Q}$. Assume $d_V(\tau_V(x),\tau_V(y)) >E$. Let $\gamma$ be a $\mc{C}_{\mf S} S$--geodesic from $\pi_S(x)$ to $\pi_S(y)$. As described in the previous paragraph, $\gamma$ intersects the $E$--neighborhood of $\rho_S^V$. Since geodesics in $ \mc{C}_{\mf S} S$ are  uniform hierarchy paths in $(\mc{C}_{\mf S} S, \mf{F})$,  the path $\mf{p}_{\pi_S(\coset{Q})}\circ \gamma = \tau_Q \circ \gamma$ is a uniform quality unparametrized quasi-geodesic in $\mc{C}_{\mf{H}}Q$.  As $\mf p$ is uniformly Lipschitz,  the projection $\mf{p}_{\pi_S(\coset{Q})}\circ\gamma$ passes through a uniform neighborhood of $\beta_{Q}^V = \mf{p}_{\pi_S(\coset{Q})}(\rho_S^V)$. Since $\mc{C}_{\mf{H}}Q$ is hyperbolic,  this implies every $\mc{C}_{\mf{H}}Q$--geodesic from $\tau_{Q}(x)$ to $\tau_{Q}(y)$ passes through a uniform neighborhood of $\beta_{Q}^V$.
	
	\textbf{Large Links:} For all  $W \in \mf{S} -\{S\}$, this axioms follows immediately from the large link axiom in $(G,\mf{S})$. Thus, we only need to verify the axiom for $S$ and domains in $\mf{Q}$.
	
	Let $x,y \in G$ and consider first $Q \in \mf{Q}$. 	
	 Since $\coset{Q}$ is hierarchically quasiconvex in $(G,\mf{S})$, there exists a gate map $\gate_{\coset{Q}} \colon G \to \coset{Q}$. Let $x' = \gate_{\coset{Q}}(x)$ and $y' = \gate_{\coset{Q}}(y)$.
	  For all $W \in \mf{S} -\{S\}$, if $W \nest Q$, then $\tau_{W} \vert_{\coset{Q}} = \pi_W \vert_{\coset{Q}}$ is coarsely onto by Lemma \ref{lem:h_nesting_and_product_regions}.
	   Hence there exists $C \geq 0$, depending only on $\mf{S}$ and $\{H_1, \dots, H_k\}$, so that $$\diam(\tau_W(x') \cup \tau_W(x)) \leq C \text{ and } \diam(\tau_W(y') \cup \tau_W(y)) \leq C$$ for each $W  \in \mf{H}_{Q}$. We can further assume that  $$\diam(\tau_{Q}(x) \cup \tau_{Q}(x')) \leq C \text{ and} \diam(\tau_{Q}(y) \cup \tau_{Q}(y'))\leq C$$ because $\pi_S \circ \gate_{\coset{Q}}$ uniformly coarsely agrees with $\tau_{Q} =\mf{p}_{\pi_S(\coset{Q})} \circ \pi_S$.
	
	By applying the large links axiom of $\mf{S}$ to $x'$ and $y'$, we produce $V_1,\dots,V_m \in \mf{S}-\{S\}$ so that $m \leq Ed_{\mc{C}_\mf{S}S}(\pi_S(x'),\pi_S(y'))+E$ and, for all $W \in \mf{S}-\{S\}$, either $W \nest V_i$ for some $i \in \{1,\dots,m\}$ or   $$d_W(\pi_W(x'),\pi_W(y')) \leq E+B.$$
	Without loss of generality, we can assume that for each $i \in \{1,\dots,m\}$, there exist $W \nest V_i$ so that $d_W(\pi_W(x'),\pi_W(y')) > E+B$. 
	In particular,  by Lemma \ref{lem:h_nesting_and_product_regions}, we may assume each $V_i$ is nested into $Q$ in $\mf{H}$.  Since $$d_W(\tau_W(x),\tau(y)) \geq d_W(\pi_W(x'),\pi_W(y')) - 2C,$$ for every $W \in \mf{H}_{Q}$, either $d_W(\tau_W(x),\tau_W(y)) \leq E+B+2C$ or $W \nest V_i$.	 
	Since $$d_{Q}(\tau_{Q}(x'),\tau_{Q}(y')) = d_{\mc{C}_\mf{S}S}(\pi_S(x'),\pi_S(y'))$$ and $$d_{Q}(\tau_{Q}(x),\tau_{Q}(y)) \geq d_{Q}(\tau_{Q}(x'),\tau_{Q}(y')) - 2C,$$  we have $m \leq Ed_{Q}(\tau_{Q}(x),\tau_{Q}(y)) +E +2C$, which completes the proof of the large links axiom for $Q \in \mf{Q}$.
	
	Now consider the domain $S$. 	Since $\mf{S}$ is maximized, $\mc{C}_{\mf{S}}S$ is the graph that has the elements of $G$ as vertices with edges between two vertices $x_1$ and $x_2$ if $x_1,x_2 \in \P_W$ for some $W\in\mf{S}-\{S\}$; see Theorem \ref{thm:maximization}\eqref{max:electrification}. Moreover, $\mc{C}_{\mf{H}}S$ is a copy of this graph $\mc{C}_{\mf{S}}S$ with additional edges between two vertices $x_1$ and $x_2$ if $x_1,x_2 \in \coset{Q}$ for some $Q \in \mf{Q}$.
	
	Let $x,y \in \mc{X}$ and let $\tau_S(x)= v_0, v_1,\dots, v_m = \tau_S(y)$ be the vertices of the $\mc{C}_\mf{H}S$--geodesic from $\tau_S(x)$ to $\tau_S(y)$. Each edge between $v_{i-1}$ and $v_{i}$ then corresponds to either a coset $\coset{Q}$ or a product region $\P_W$. Let $V_i$ be the elements of $\mf{H}$ corresponding to the edge between $v_{i-1}$ and $v_i$.  If $V_i \in \mf{S}$, let $U_i$ be a container for $V_i$ in $\mf{S}$ (note, $\mf{S}_{V_i}^\perp \neq \emptyset$ because $\mf{S}$ is maximized). By construction $2m =2d_{\mc{C}_\mf{H}S}(\tau_S(x),\tau_S(y))$. We will show that for every $W \in \mf{H}- \{S\}$, either $W$ is nested into some $V_i$ or $U_i$, or $d_W(\tau_W(x),\tau_W(y))$ is uniformly bounded. 
	
	Since we have already verified that $\mf{H}$ satisfies the bounded geodesic image axiom, let $C \geq 0$ be the maximum of the constant from the bounded geodesic image axiom for $\mf{H}$ and the bound on the diameters of $\beta_S^W$ for each $W \in \mf{H} -\{S\}$. Let $W \in \mf{H} -\{S\}$. Since $\mf{S}$ has the bounded domain dichotomy, we can assume $W \in \mf{S}^{\infty}$. 
	If $d_{\mc{C}_\mf{H}S}(v_i,\beta_S^W) > C+3$ for all $v_i \in \{v_0,\dots,v_m\}$, then $d_W(\tau_W(x),\tau_W(y)) \leq C$ by the bounded geodesic image axiom. 
	Otherwise, let $j$ be the minimal element of $\{0,\dots,m\}$ so that $d_{\mc{C}_\mf{H}S}(v_j,\beta_S^W) \leq C+3$. By construction, if $i < j$ or $i \geq j+3C +6$, then $d_{\mc{C}_\mf{H}S}(v_i,\beta_S^W) >C+3$. 
	Hence, the bounded geodesic image axiom says \[ \diam(\tau_W(x) \cup \tau_W(v_i)) \leq C \text{ for } i < j \] and \[\diam(\tau_W(y) \cup \tau_W(v_i)) \leq C \text{ for } i \geq j+3C+6.\]  
	Thus we have \[d_W(\tau_W(x),\tau_W(y)) \leq \sum_{i=j}^{j+3C +6} \diam(\tau_W(v_i) \cup \tau_W(v_{i+1}) ) +2C.\] 
	Hence, there exists $C' \geq 0$ depending only on $C$ and $\mf{S}$ so that if $d_W(\tau_W(x),\tau_W(y)) >C'$, then  for at least one $i \in\{j,\dots,j+3C+6\}$, we have $\diam(\tau_W(v_i) \cup \tau_W(v_{i+1} ) >3E + B$.

	If $V_{i+1} = Q \in \mf{Q}$, then $v_i$ and $v_{i+1}$ are in the coset $\coset{Q}$, implying $\diam(\tau_W(\coset{Q})) \geq B$. By Lemma \ref{lem:h_nesting_and_product_regions}, this implies $W \nest Q = V_{i+1}$. On the other hand,  if $V_{i+1} \in \mf{S}$, then  $v_i,v_{i+i} \in \P_{V_{i+1}}$. Hence $\diam(\tau_W(\P_{V_{i+1}}))> 3E$, which implies $W \nest V_{i+1}$ or $W \perp V_{i+1}$ by Proposition \ref{prop:product_regions}\eqref{P_prop:large_projection}.    Thus, for all $W \in \mf{H}- \{S\}$, either $d_W(\tau_W(x),\tau_W(y)) <C'$ or there is $i \in \{1,\dots,m\}$ so that $W \in \mf{H}_{V_{i+1}} \cup \mf{H}^\perp_{V_{i+1}}$. Since $U_{i+1}$ is a container for $V_{i+1}$ when $\mf{H}^\perp_{V_{i+1}} \neq \emptyset$, this means $W$ is nested into either $V_{i+1}$ or $U_{i+1}$ whenever  $d_W(\tau_W(x),\tau_W(y)) >C'$.

	\textbf{Consistency:} Because many of the relative projections in $\mf{H}$ are the same as the relative projections in either $\mf{S}$ or $\mf{F}$, we only need to verify the first inequality for $V \in \mf{S} -\{S\}$ and $Q \in \mf{Q}$ with $V \trans Q$. Suppose $x \in \mathcal X$ with $d_V(\tau_V(x), \beta_V^{Q}) >E$. Let $y$ be any point in $\coset{Q}$. Since $\beta_V^{Q} = \tau_{V}(\coset{Q})$, we have $d_V(\tau_V(x), \tau_V(y)) = d_V(\pi_V(x),\pi_V(y)) >E$.  By the bounded geodesic image axiom in $\mf{S}$, this implies every $\mc{C}_{\mf S}S$--geodesic from $\pi_S(x)$ to a point in $\pi_S(\coset{Q})$ passes $E$--close to $\rho_S^V$. Hence $\mf{p}_{\pi_S(\coset{Q})}(\rho_S^V) = \beta_{Q}^V$ is uniformly close to  $\mf{p}_{\pi_S(\coset{Q})}(x)=\tau_{Q}(x)$, and the first inequality holds.
	
	For the second inequality, we only need to check the case where $V \propnest Q$ and  there is a domain $W \in \mf{H}$ so that either $Q \propnest W$ or $Q  \trans  W$ and $W \not \perp V$.  By Lemma \ref{lem:h_nesting_and_product_regions},  $\P_V$ is contained in a regular neighborhood of $\coset{Q}$ as $V \nest Q$. Now, the only way for $Q \propnest W$ is if $W = S$. In this case, $\beta_S^V = \rho_S^V$ and $\beta_S^{Q} = \tau_S(\coset{Q})$ are uniformly close in $\mc{C}_{\mf H} S$ because   $\P_V$ is contained in a regular neighborhood of $\coset{Q}$. If instead $Q  \trans  W$, then $\tau_W(\P_V)$ is contained in a uniform neighborhood of $\beta_{W}^{Q} = \tau_W(\coset{Q})$. Since $\rho_W^V = \beta_W^V$ is uniformly close to $\pi_W(\P_V) = \tau_W(\P_V)$ this implies $\beta_W^V$ and $\beta_W^{Q}$ are uniformly close. 
	
	\textbf{Partial Realization:} Since $\mf{H}$ has no new orthogonality, we only need to verify this axiom for a single domain in $\mf{Q}$. However, the definition of  $\tau_*$  plus the relations on $\mf{H}$ make this axiom   automatically satisfied for these domains.	
\end{proof}

\section{The boundary of relatively hyperbolic $G$--HHSs}\label{sec:boundary_rel_hyp}

In this section, we characterize the simplicial structure of the boundaries of relatively hyperbolic $G$--HHSs. We start with the more straightforward part, which describes the boundary of a relatively hyperbolic $G$--HHS.   We will then show that whenever this description of the boundary of a $G$--HHS holds, the group is relatively hyperbolic (Theorem~\ref{thm:boundary_criteria_for_rel_hyp}).  Recall that $\Lambda(\cdot)$ denotes the limit set of a subset of an HHS in the HHS boundary.

\begin{thm}\label{thm:boundary_of_rel_hyp_HHG}
	Let $(G,\mf{S})$ be a $G$--HHS. If $G$ is hyperbolic relative to a
	finite collection of infinite index subgroups $\{H_1,\dots,
	H_k\}$, then there exist disjoint subcomplexes $\Lambda_1,\dots,
	\Lambda_k$ of $\partial_\Delta(G,\mf{S})$ so that
	\begin{enumerate}
		\item each $H_i$ is hierarchically quasiconvex and $\Lambda_i$ is the limit set of $H_i$ in $\partial(G,\mf{S})$;
		\item for all $1 \leq i < j\leq k$ and $g,h \in G$ we have $ g\Lambda_i \cap h\Lambda_j =\emptyset$ unless $i=j$ and $g^{-1}h \in H_i$; and
		\item $\partial_{\Delta}(G,\mf{S}) - G \cdot \left( \bigsqcup\limits_{i=1}^{k} \Lambda_i \right)$ is  a non-empty set of isolated vertices.
	\end{enumerate}
\end{thm}

The proof of Theorem \ref{thm:boundary_of_rel_hyp_HHG} will rely on the following classification of the limit sets of
hyperbolically embedded subgroups in the HHS boundary for the
structure $\mf{H}$ from Construction \ref{con:HHG+hyp_embedded}.

\begin{lem}\label{lem:limit_set_support}
		Let $(G,\mf{S})$ be a maximized $G$--HHS and $\{H_1,\dots,H_k\}$ be a hyperbolically embedded collection of subgroups. Let $\mf{H}$ be the $G$--HHS structure from Construction \ref{con:HHG+hyp_embedded} such that $\mf{Q} \subseteq \mf{H}$ is the set indexing the cosets of the $H_i$. For all $Q \in \mf{Q}$, a point $p \in \partial_\Delta(G,\mf{H})$ is in the limit set of the coset $\coset{Q}$ if and only if every element of  $\supp(p)$ is nested into $Q$ in $\mf{H}$.
\end{lem}
	
	\begin{proof}
		We use the notation of Construction \ref{con:HHG+hyp_embedded} for $\mf{S}$ and $\mf{H}$.
		
		 If $p\in \Lambda(\coset{Q})$ and $W\in\supp(p)$, then $\diam(\pi_W(\coset{Q}))=\infty$.  By Lemma~\ref{lem:h_nesting_and_product_regions}\eqref{item:uniform_onto-nested}, this implies that $W\nest Q$.

		For the other direction, recall that  $W \propnest Q$ implies $W \in \mf{S}$ and $\P_W$ is contained in a regular neighborhood of $\coset{Q}$ by Lemma \ref{lem:h_nesting_and_product_regions}. In particular, $\pi_V \vert_{\coset{Q}}$ is coarsely onto for all $V \in \mf{S}_W \cup \mf{S}_W^\perp$, and so $\partial\mc C_{\frak H}V\subseteq \Lambda(\coset{Q})$. Thus, if  $\supp(p) =\{W_1,\dots, W_m\}$ and each $W_{\ell}$ is nested into $Q$, then $p \in \Lambda(\coset{Q})$ because the join of all the $\partial \mc{C}_\mf{H}W_{\ell}$ is contained in $\Lambda(\coset{Q})$.
	\end{proof}

\begin{proof}[Proof of Theorem \ref{thm:boundary_of_rel_hyp_HHG}]
	If $G$ is hyperbolic relative to $\{H_1,\dots, H_k\}$, then  $\{H_1,\dots, H_k\}$ is  a hyperbolically embedded collection of subgroups. In particular, each $H_i$ is hierarchically quasiconvex in every $G$--HHS structure for $G$ by Theorem \ref{thm:hyperbolically_embedded_in_HHGs}. 
	
	If $\mf{T}$ is the maximization of $\mf{S}$, then Theorem \ref{thm:maximization_invariance}\eqref{invar:boundary} says that  there is a homeomorphism  $$\Phi \colon G \cup  \partial (G,\mf{S}) \to G\cup \partial (G,\mf{T})$$  that restricts to the identity on $G$ and is both a homeomorphism and simplicial isomorphism on the boundary.   In particular, the limit set of each $H_i$ in $\partial(G,\mf{S})$ is mapped homeomorphically by $\Phi$ to the limit set of $H_i$ in $\partial(G,\mf{T})$. Hence, we can assume $\mf{S}$ is a maximized HHS structure.
	
	Now that $\mf{S}$ has been maximized, we can apply   Theorem \ref{thm:adding_in_hyperbolically_embedded_subgroups} to produce the $G$--HHS structure $\mf{H}$ for $G$ as described in Construction \ref{con:HHG+hyp_embedded}. By Corollary \ref{cor:adding_hyp_embedded+boundary}, there is a homeomorphism $$\Psi \colon G \cup  \partial (G,\mf{H}) \to G\cup \partial (G,\mf{S})$$  that restricts to the identity on $G$ and is both a homeomorphism and simplicial isomorphism on the boundary. As before, the limit set of each $H_i$ in $\partial(G,\mf{H})$ is mapped homeomorphically by $\Psi$ to the limit set of $H_i$ in $\partial(G,\mf{S})$. Taken together, this means it suffices to prove the result for $\partial (G,\mf{H})$ instead of $\partial (G,\mf{S})$.

	Let $\Lambda_i$ be the limit set of $H_i$ in $\partial (G,\mf{H})$, so that  $g\Lambda_i$ is the limit set of the coset $gH_i$ in $\partial (G,\mf{H})$. Let $\mf{Q} \subseteq \mf{H}$ be the set indexing the cosets of the $H_i$ as in Construction \ref{con:HHG+hyp_embedded}. We will continue to use $\coset{Q}$ to denote the coset in $G$ indexed by $Q$.	We use Lemma \ref{lem:limit_set_support} to verify the conclusions of Theorem \ref{thm:boundary_of_rel_hyp_HHG}.
	
	To see that each $\Lambda_i$ is a subcomplex, let $Q \in \mf{Q}$ with $\coset{Q} = H_i$,  and let $p,q$ be vertices of $\Lambda_i$ that are joined by an edge, $e_{pq}$, of $\partial_\Delta(G,\mf{H})$. This means there are domains $V,W\in\frak H$ such that $\supp(p) = \{W\}$, $\supp(q) =\{V\}$, and $W \perp V$.  Since no element of $\mf{Q}$ is  orthogonal to any other domain, we have $W,V \in\mf{S}$.  Thus $W,V \propnest Q$ by Lemma \ref{lem:limit_set_support}. The support of any point on the edge $e_{pq}$ is contained in $\{W,V\}$. Thus, Lemma \ref{lem:limit_set_support} says $e_{pq} \subseteq \Lambda_i$. 
	
	For the second item, Lemma \ref{lem:limit_set_support} says  that if $ g\Lambda_i \cap h\Lambda_j \neq \emptyset$, then there is $W \in \mf{H}_{Q} \cap  \mf{H}_{R}$, where $\coset{Q} = gH_i$ and $\coset{R} = hH_j$. However, this would imply $\P_W$ is contained in a regular neighborhood of both $\coset{Q}= gH_i$ and $\coset{R} = hH_j$ by Lemma \ref{lem:h_nesting_and_product_regions}. Since $\diam(\P_W) = \infty$ because $\mf{S}$ is maximized, this implies $i=j$ and $g^{-1}h \in H_i$ by Lemma \ref{lem:intersection_of_almost_malnormal}.

	Finally, because $\mf{H}$ has orthogonality isolated by $\mf{Q}$ (Corollary \ref{cor:rel_hyp_structure}), Lemma \ref{lem:limit_set_support} says every $p \in \partial_{\Delta}(G,\mf{H})$ is either in some $g\Lambda_i$ or has $\supp(p)=\{S\}$, where $S$ is the $\nest$--maximal element of $\mf{S}$. Hence, the set  $$\partial_{\Delta}(G,\mf{S}) - G \cdot \left( \bigsqcup\limits_{i=1}^{k} \Lambda_i \right)$$ is a collection of isolated vertices in $\partial_\Delta(G,\mf{H})$ because each point in it has support $\{S\}$.	
\end{proof}

We now show that the only way for the boundary of a $G$--HHS to decompose as described in Theorem \ref{thm:boundary_of_rel_hyp_HHG} is for the group to be relatively hyperbolic.

\begin{thm}\label{thm:boundary_criteria_for_rel_hyp}
	Let $(G,\mf{S})$ be a $G$--HHS. Let $\Lambda_1,\dots, \Lambda_k$ be disjoint  subcomplexes of $\partial_\Delta(G,\mf{S})$, and let $H_i = \stab_G(\Lambda_i)$. Suppose
	\begin{enumerate}
		\item \label{item:infinite_index} each  $H_i$ is hierarchically quasiconvex and has infinite index in $G$;
		\item \label{item:limit_simplices} for each $i$, $\Lambda_i$  is the limit set of $H_i$ in $\partial(G,\mf{S})$;
		\item \label{item:distinct_cosets} for all $1 \leq i < j\leq k$ and $g,h \in G$, we have $ g\Lambda_i \cap h\Lambda_j =\emptyset$ unless $i=j$ and $g^{-1}h \in H_i$; and 
		\item \label{item:disjoint_union} $\partial_{\Delta}(G,\mf{S}) - G \cdot \left( \bigsqcup\limits_{i=1}^{k} \Lambda_i \right)$ is  a non-empty set of isolated vertices. 
	\end{enumerate}
	Then $G$ is hyperbolic relative to the subgroups $H_1, \dots, H_k$.
\end{thm}

\begin{proof}
	First we explain why we can assume $\mf{S}$ is maximized.  Let
	$\mf{T}$ be the maximization of $\mf{S}$.  By Theorem
	\ref{thm:maximization_invariance}\eqref{invar:HQC}, each $H_i$ is
	hierarchically quasiconvex with respect to both $\mf{S}$ and
	$\mf{T}$, and Theorem
	\ref{thm:maximization_invariance}\eqref{invar:boundary} provides a
	map $\Phi \colon G \cup \partial(G,\mf{S}) \to G \cup
	\partial(G,\mf{T})$ that is the identity on $G$ and both a
	homeomorphism and a simplicial isomorphism on the boundary.  In
	particular, $(G,\mf{T})$ satisfies the hypotheses of Theorem
	\ref{thm:boundary_criteria_for_rel_hyp} with respect to the
	complexes $\Phi(\Lambda_i)$.  Hence, without less of generality 
	we may assume $\mf{S}$ is already 
	maximized.

	The bulk of our proof will be showing that  $\{H_1, \dots, H_k\}$ is a hyperbolically embedded collection of subgroups. This will allow us to use Theorem \ref{thm:adding_in_hyperbolically_embedded_subgroups} to create an HHS structure for $G$ with isolated orthogonality. 
	
	For the remainder of the proof, let $\supp(g\Lambda_i)$ denote the union of the support sets of all the elements of $g\Lambda_i$, where $g \in G$.\\
	
\noindent {\bf Step 1:}	

		The set $\{H_1,\dots, H_k\}$ is an almost malnormal collection of subgroups.

		Suppose $H_i \cap g H_j g^{-1}$ is infinite. There then exists an infinite sequence $$(h_n) \subseteq H_i \cap g H_j g^{-1}$$ so that $(h_n)$ converges to a point in $\partial (G,\mf{S})$.  Because each $h_n$ is in $gH_jg^{-1} =\stab_G(gH_j)$, we have $h_ng \in g H_j$ for each $n$. Since $d_G(h_n,h_ng) = d_G(e,g)$, we have that $(h_n)$ and $(h_n g)$ converge to the same point in $\partial (G,\mf{S})$ by Lemma \ref{lem:bounded_difference_convergence}.  The limit of $(h_n)$ is in $\Lambda_i$, while the limit of $(h_n g)$ is in $g \Lambda_j$, so by Hypothesis \eqref{item:distinct_cosets}, we must have $i=j$ and $g\in P_i$. \\

	\noindent{\bf Step 2:}  Each $H_i$ is uniformly strongly quasiconvex. 

For this step we need several auxiliary claims.
	
	\begin{claim}\label{claim:suppport_closed_under_orthogonality}
		Suppose  $W \in \supp(\Lambda_i)$  is not $\nest$--maximal in $\mf{S}$. Then $\partial\mc{C} W$  is contained in $\Lambda_i$, as is $\partial \mc{C}V$ for any $V \in \mf{S}^{\infty}$ with $W \perp V$.
	\end{claim}
	
	\begin{proof}
		If  $W \in \supp(\Lambda_i)$, there is a point in $\Lambda_i$ whose support set includes $W$. Such a point is in a simplex that has a vertex $p$ with $\supp(p) = \{W\}$. Since $\Lambda_i$ is a subcomplex, the vertex $p$ must also be in $\Lambda_i$. Because $\mf{S}$ is maximized and $W$ is not $\nest$--maximal, there must exists $V \in \mf{S}^\infty$ with $V \perp W$. Let $q$ be any point in $\partial \mc{C}V$. The edge in $\partial_\Delta (G,\mf{S})$ between $p$ and $q$ is contained in some $g\Lambda_j$ by Hypothesis (\ref{item:disjoint_union}). Since this implies $p \in g\Lambda_j \cap \Lambda_i$, we must have $g\Lambda_j = \Lambda_i$ by Hypothesis (\ref{item:distinct_cosets}). Hence, $q \in \Lambda_i$ as well. Thus $\partial \mc{C}V \subseteq \Lambda_i$. By repeating the argument with the roles of $W$ and $V$ reversed we have that $\partial \mc{C}W \subseteq \Lambda_i$ as well. 
	\end{proof}
	
	\begin{claim}\label{claim:project_onto_support}
		If $W \in \supp(\Lambda_i)$ is not $\nest$--maximal in $\mf{S}$, then $\pi_W \vert _{H_i}$ is uniformly coarsely onto.
	\end{claim}
	
	\begin{proof}
		By Claim \ref{claim:suppport_closed_under_orthogonality}, if $W \in  \supp(\Lambda_i)$ is not $\nest$--maximal, then $\partial\mc{C} W \subseteq \Lambda_i$. Since $\Lambda_i$ is the limit set of $H_i$,  $\partial \mc{C}W$ must be the limit set of $\pi_W(H_i)$ in $\mc{C}W$ (Lemma \ref{lem:convergence_of_interior_points}). Since $\pi_W(H_i)$ is uniformly quasiconvex in $\mc{C}W$, the only way for this to happen is if some uniform neighborhood of $\pi_W(H_i)$ covers $\mc{C}W$.
	\end{proof}
	
	\begin{claim}\label{claim:support_and_product_regions}
		There exists $\nu \geq 0$ so that for any $W \in \supp(\Lambda_i)$,  if $W$ is not $\nest$--maximal in $\mf{S}$, then the product region $\P_W$ is contained in the $\nu$--neighborhood of $H_i$.
	\end{claim}
	
	\begin{proof}
		Let $W \in \supp(\Lambda_i)$ be non-$\nest$--maximal in $\mf{S}$. Because $\mf{S}$ is maximized, $\mf{S}_W \cap \mf{S}^\infty$ and  $\mf{S}_W^\perp \cap \mf{S}^\infty$ are both non-empty. Let $U \in \mf{S}_W \cap \mf{S}^\infty$ and $V \in \mf{S}_W^\perp\cap \mf{S}^\infty$. Since $W \perp V$ and $U \nest W$, we have $U \perp V$. Thus, by applying Claim \ref{claim:suppport_closed_under_orthogonality} twice, we have both $V,U \in \supp(\Lambda_i)$. By Claim \ref{claim:project_onto_support}, both $\pi_V$ and $\pi_U$ are uniformly coarsely onto when restricted to $H_i$. Since $H_i$ is hierarchically quasiconvex,  this implies $\P_W$ is contained in a uniform neighborhood of $H_i$ by Proposition \ref{prop:product_regions}\eqref{P_prop:minimal_HQC}.
	\end{proof}

	We are now ready to show that each $H_i$ is uniformly strongly quasiconvex. 	Let $S$ be the $\nest$--maximal element of $\mf{S}$.
		
		Since each $H_i$ is hierarchically quasiconvex, it suffices to show that each $H_i$ has the orthogonal projection dichotomy (Definition \ref{defn:ortho_proj_dich}). 
		In light of Claims \ref{claim:suppport_closed_under_orthogonality} and \ref{claim:project_onto_support} and the bounded domain dichotomy of $\mf{S}$, the subgroup $H_i$ will have the orthogonal projection dichotomy if the projection of $H_i$  to every element of $\mf{S}^{\infty}- (\supp(\Lambda_i) \cup \{S\})$ has uniformly bounded diameter. For the purposes of contradiction, suppose not. We can then find a sequence  of points $(x_n)$ in $H_i$ and a collection of unbounded domains $W_n \in \mf{S}^{\infty}- (\supp(\Lambda_i) \cup \{S\})$ so that $d_{W_n}(e,x_n) \to \infty$ as $n \to \infty$.
		
		Because $\mf{S}$ is maximized, $W_n \neq S$ implies there are domains $V_n \in \mf{S}^{\infty}$ with $W_n \perp V_n$. This means that for each $n$, the join $\partial \mc{C}W_n \star \partial \mc{C}V_n$ is a subcomplex of $\partial_\Delta(G,\mf{S})$. Hence, by Hypothesis \eqref{item:disjoint_union}, there is $g_n \in G$ so that $W_n \in \supp(g_n \Lambda_{j_n})$ for each $n$. By Hypothesis \eqref{item:limit_simplices}, either $j_n \neq i$  or $g_n \not \in H_i$ for each $n \in \mathbb{N}$.
		
		By \cite[Proposition 4.24]{RST_Quasiconvexity}, there exists constants $\lambda$, $\nu$, and $D$ depending only on $\mf{S}$, so that whenever $d_{W_n}(e,x_n) \geq D$, there is a $\lambda$--hierarchy path $\gamma_n$ connecting $e$ and $x_n$ with a subinterval $\alpha_n$ so that
		\begin{itemize}
			\item $\alpha_n$ is contained in the $\nu$--neighborhood of $\P_{W_n}$; and
			\item the diameter of $\alpha_n$ is bounded below by $\nu^{-1}\cdot d_{W_n}(e,x_n) - \nu$.
		\end{itemize}
		Because $d_{W_n}(e,x_n) \to \infty$, we can assume $n$ is large enough so that  $d_{W_n}(e,x_n) \geq D$, and hence such a hierarchy path $\gamma_n$ exists.

		Since $H_i$ is hierarchically quasiconvex and $x_n \in H_i$, the hierarchy path $\gamma_n$ stays uniformly close to $H_i$ by Proposition \ref{prop:HQC_and_hp}.  
		Because $W_n \in \supp(g_n \Lambda_{j_n})$, the product region $\P_{W_n}$ is also contained in some uniform neighborhood of $g_nH_{j_n}$ by Claim \ref{claim:support_and_product_regions}. Hence, there is a uniform constant $\nu'$ so that the interval $\alpha_n$ is  contained in  $$\mc{N}_{\nu'}(H_i) \cap \mc{N}_{\nu'}(g_n H_{j_n})$$ for each $n$. 
		
		It follows that there exists $h_n \in H_i$ so that each coset $h_n^{-1}g_n H_{j_n}$ is uniformly close to the identity $e \in H_i$. Since either $j_n \neq i$  or $g_n \not \in H_i$ for each $n \in \mathbb{N}$, we have $H_i \neq h_n^{-1}g_n H_{j_n}$ for each $n \in \mathbb{N}$.  Corollary 3.13 of \cite{coarse_helly} proved that hierarchically quasiconvex subgroups have bounded packing, hence $\{h_n^{-1}g_n H_{j_n}\}$ must be a finite collection of cosets. The  intersection of $\mc{N}_{\nu'}(H_i)$ and $\mc{N}_{\nu'}(h_n^{-1}g_{n}H_{j_n})$ contains $h_n^{-1}\alpha_n$, which gets arbitrarily large as $n \to \infty$. Thus, there is some $n_0$ so that $\mc{N}_{\nu'}(H_i)$ and $\mc{N}_{\nu'}(h_{n_0}^{-1}g_{n_0}H_{j_{n_0}})$ have infinite diameter  intersection and $H_i \neq h_{n_0}^{-1}g_{n_0}H_{j_{n_0}}$. However, this violates the fact that  $\{H_1,\dots, H_k\}$ is almost malnormal (Lemma \ref{lem:intersection_of_almost_malnormal}). Thus, there must a uniform bound on diameter of $\pi_W(H_i)$ for each  $W\in \mf{S}^{\infty}- (\supp(\Lambda_i) \cup \{S\})$, as desired.\\

	\noindent {\bf Step 3:} $G$ is hyperbolic relative to  $\{H_1, \dots, H_k\}$. 
 
	Since  $\{H_1, \dots, H_k\}$ is an almost malnormal collection of strongly quasiconvex subgroups, it is hyperbolically embedded in $G$ by Theorem \ref{thm:hyperbolically_embedded_in_HHGs}. Let $\mf{H}$ be the $G$--HHS structure  from Theorem \ref{thm:adding_in_hyperbolically_embedded_subgroups} that adds the cosets of the $H_i$ to $\mf{S}$. We will show that $\mf{H}$ has orthogonality isolated by $\mf{Q}$, the set indexing the cosets of  $H_1, \dots, H_k$. As in Construction \ref{con:HHG+hyp_embedded}, for each $Q\in \mf{Q}$, let $\coset{Q}$ denote the coset in $G$ indexed by $Q$.
	
	Suppose $V, W \in \mf{H}$ with $V \perp W$. Since the only orthogonal elements of $\mf{H}$ come from $\mf{S}$, we have $V,W \in \mf{S} - \{S\}$. Because $\mf{S}$ is maximized, there exist $V' \nest V$ and $W' \nest W$ with $V',W' \in \mf{S}^\infty$. Since $V'$ and $W'$ are orthogonal, the join $\mc{C}V' \star \mc{C}W'$ must be contained in some $g\Lambda_i$. Thus $\pi_{V'}\vert_{gH_i}$ and $\pi_{W'} \vert_{gH_i}$ are both coarsely onto. Since $V' \perp W$ and $W'\perp V$, this implies $V,W \nest Q$ where $Q$ is the element of $\mf{Q}$ with $\coset{Q} = gH_i$. 
	
	Now suppose there is $V\in\mf{H}$, and $Q,R \in\mf{Q}$ with  $V \nest Q$ and $V \nest R$. By Lemma \ref{lem:h_nesting_and_product_regions}, the infinite diameter product region $\P_V$ is then contained in a uniform neighborhood of both $\coset{Q}$ and $\coset{R}$ in $G$. By Lemma \ref{lem:intersection_of_almost_malnormal}, this can only happen if $\coset{Q} = \coset{R}$. Hence, $Q = R$.
	
	Since $ \mf{Q}$ does not contain the $\nest$--maximal element of
	$\mf{H}$ by construction, the above two paragraphs show that
	$\mf{Q}$ isolates the orthogonality of $\mf{H}$, making $G$
	hyperbolic relative to the product regions of the elements of
	$\mf{Q}$, by Theorem
	\ref{thm:isolated_orthogonality_implies_rel_hyp}.  However, for
	each $Q \in \mf{Q}$, the product region for $Q$ in $\mf{H}$ is
	within finite Hausdorff distance of the coset $\coset{Q}$ by
	Remark \ref{rem:products_in_h}.  Hence, by \cite[Theorem 
	1.5]{Drutu_Rel_Hyp_is_Geometric}, the group $G$ is hyperbolic
	relative to $\{H_1,\dots, H_k\}$.
	
\end{proof}

\section{The Bowditch boundary}\label{sec:bowditch_boundary}

If a finitely generated group $G$ is hyperbolic relative to a collection of subgroups $\mc{P}$, then the Gromov boundary of the hyperbolic space $\cusp(G,\mc{P})$ is called the \emph{Bowditch boundary} of the pair $(G,\mc{P})$.  In this section, we prove the following theorem, which establishes the Bowditch boundary of a relatively hyperbolic $G$--HHS as a quotient of the HHS boundary.

\begin{thm}\label{thm:bowditch_boundary}
	Let $(G,\mf{S})$ be a $G$--HHS, and suppose $G$ is hyperbolic relative to a finite collection of subgroups $\mc{P}$. There is a quotient map $\Psi \colon  \partial(G,\mf{S}) \to \partial \cusp(G,\mc{P})$ so that for distinct $p,q \in \partial(G,\mf{S})$, we have $\Psi(p) = \Psi(q)$ if and only if  there exists $g \in G$ and $H \in \mc{P}$ so that $p$ and $q$ are both in the limit set of $gH$ in $\partial(G,\mf{S})$. Moreover, the inclusion  $G \to \cusp(G,\mc{P})$ extends continuously to $\Psi$.
\end{thm}

Before proving Theorem \ref{thm:bowditch_boundary}, we will collect some additional preliminary results on the distances in combinatorial horoballs (Section \ref{subsec:horoballs}) and on the topology on the HHS boundary (Section \ref{subsec:boundary_top}). We will then prove Theorem \ref{thm:bowditch_boundary} in the special case where $\mf{S}$ has isolated orthogonality (Section \ref{subsec:isolated_orthogonality_case}). Finally, we reduce the general case to the case of isolated orthogonality using Corollary \ref{cor:adding_hyp_embedded+boundary}, which adds hyperbolically embedded subgroups to the structure without changing the boundary (Section \ref{subsec:proof_of_bowditch_boundary}).

\subsection{Distances in combinatorial horoballs}\label{subsec:horoballs}

The following result of Mackay and Sisto provides a formula for 
computing distances in combinatorial horoballs.

\begin{lem}[{\cite[Lemma 3.2]{MacKay_Sisto_boundary_maps}}]\label{lem:horoball_distance_formula}
	Let $\Gamma$ be a graph  and $\mc{H}(\Gamma)$  the combinatorial horoball over $\Gamma$. There exist $c \geq 0$ so that for all $(x,n),(y,m) \in \mc{H}(\Gamma)$, we have 
	\[d_{\mc{H}(\Gamma)}((x,n),(y,m)) \stackrel{1,c}{\asymp} 2\log\left(d_\Gamma(x,y) e^{-\max\{n,m\}}+1\right) + |m -n|. \]
\end{lem}

Using this distance formula, we can show that as points in the base of the horoball move farther away from the basepoint they move closer to the single boundary point at infinity.

\begin{lem}\label{lem:horoball_base_distance_to_boundary}
	Let $Y$ be a $(\lambda,\lambda)$--quasi-geodesic space. Let $N$ be a  $10\lambda$--net in $Y$ and  $\Gamma$ be an approximation graph for $Y$ with vertex set $N$. Let $\mc{H}(Y)$ be the combinatorial horoball obtained by attaching each vertex $(v,0) \in \mc{H}(\Gamma)$ to $v \in N \subseteq Y$ by an edge of length 1. Let $\xi$ be the single boundary point of the hyperbolic space $\mc{H}(Y)$. There is a increasing function $f \colon [0,\infty) \to [0,\infty)$, depending only on $\lambda$, so that $$d_{Y}(x_0,x) \geq f(r) \implies \gprel{x}{\xi}{x_0} >r,$$ where the Gromov product is in $\mc{H}(Y)$. In particular, $d_{Y}(x_0,x) \geq f(r)$ implies  $x$ is contained in the basis neighborhood $M(r;\xi)$ for the compactification $\overline{\mc{H}(Y)}$ with basepoint $x_0$.
\end{lem}

\begin{proof}
	For each $n \in \mathbb{Z}_{\geq 0}$, let $x_n$ be the vertex
	$(x_0,n) \in \mc{H}(\Gamma)$.  Because $\Gamma$ is quasi-isometric
	to $Y$ and $\mc{H}(\Gamma)$ is quasi-isometric to $\mc{H}(Y)$,
	each with constants depending only on $\lambda$, it suffices to
	prove the result for $\mc{H}(\Gamma)$.  In fact, this is the only
	source for the dependency of $f$ on $\lambda$.

	By definition,
	$\gprel{x}{\xi}{x_0}$ is the limit of $\gprel{x}{x_n}{x_0}$ as $n
	\to \infty$.  Letting $c\geq 0$ be the constant from Lemma
	\ref{lem:horoball_distance_formula}, which we apply to three 
	different pairs of points, we have:
	\begin{align*}
		d_{\mc{H}(\Gamma)}(x,x_0)  &\geq 2\log(d_\Gamma(x,x_0)+1) -c; \\
		d_{\mc{H}(\Gamma)}(x_n,x_0) &\geq 2\log(d_\Gamma(x_0,x_0) e^{-n} +1) +n -c =n-c;\\
		d_{\mc{H}(\Gamma)}(x_n,x) &\leq 2\log(d_\Gamma(x,x_0)e^{-n} +1) +n+c.
	\end{align*}
	\noindent  Which implies:
	\begin{align*}
		2\gprel{x}{x_n}{x_0} &= d_{\mc{H}(\Gamma)}(x,x_0) + d_{\mc{H}(\Gamma)}(x_n,x_0) - d_{\mc{H}(\Gamma)}(x_n,x)\\
		&\geq 2\log(d_\Gamma(x,x_0)+1)-c+(n-c)-\left( 2\log(d_\Gamma(x_0,x)e^{-n}+1)+n-c\right).
	\end{align*}
	Hence, for any $\varepsilon>0$, there exists a sufficiently large $n$ such that 
	\[
	2(x\mid \xi)_{x_0}\geq 2\log(d_\Gamma(x,x_0)+1)-\varepsilon -3c.
	\]
	Therefore,  $\gprel{x}{\xi}{x_0}$ is  bounded below by a function of $d_{\Gamma}(x_0,x)$ as desired.
\end{proof}

\subsection{Open sets in the HHS boundary}\label{subsec:boundary_top}
We now describe a way to construct open sets around points in the HHS boundary.  For each $p \in \partial(\mc{X},\mf{S})$ and $r \geq 0$, we will define a set $\mc{A}_r(p)$. While the sets  $\mc{A}_r(p)$ may not be open themselves, they are constructed so that they each contain a element of the basis of the topology on $\partial (\mc{X},\mf{S})$.

 To define $\mc{A}_r(p)$ we need to extend the HHS projection maps to points in the boundary.

\begin{defn}\label{def:boundaryprojections}
	Fix a point $q = \sum_{W \in \supp(q)}a_W q_W \in \partial (\mc X,\s)$.  For each $U\in \s$ such that there exists $W\in \supp(q)$ with $U\not\perp W$, we define the \emph{boundary projection $\partial \pi_U(q)$} of $q$ into $\mc CU$ as follows.
	\begin{itemize}
		\item If $W = U$,  define $\partial \pi_U(q) := q_U=q_W$.
		\item If $W\propnest U$ or $W\trans U$,  let $\mc{V} = \{V \in \supp(q) : V \trans U \text{ or } V \propnest U \}$, and define $$\partial \pi_U(q): = \bigcup_{V \in \mc{V}} \rho^V_U.$$
		\item If $W\sqsupsetneq U$,  we will use the map $\rho_U^W \colon \mc{C}W \to \mc{C}U$ from Lemma \ref{lem:downward_rhos} to define $\partial \pi_U(q)$. Let $\sigma\geq 0$ be the constant so that any two $(1,20E)$--quasi-geodesics with the same endpoints in a $E$--hyperbolic metric space are $\sigma$--close together. Let $Z\subseteq \mc C W$ be the set of  all points on all $(1,20E)$--quasi-geodesics from a point in $\rho^U_W\in\mc CW$ to $q_W\in\partial \mc CW$ that are at distance at least $2E+\sigma$ from $\rho^U_W$.  Define $$\partial \pi_U(q): =\rho^W_U(Z).$$
	\end{itemize} 
\end{defn}

The definition of $\mc{A}_r(p)$ is divided into two parts depending on the relationship with the support of $p$.

\begin{defn}\label{def:remote}
	Let $(\mc X,\mf{S})$ be a hierarchically hyperbolic space, and let $p\in \partial(\mc X,\s)$.  A point $q\in \partial (\mc X,\s)$ is \emph{remote to $p$} if:
	\begin{enumerate}
		\item $\supp(p)\cap \supp(q)=\emptyset$; and 
		\item for all $Q\in \supp(q)$, there exists $P\in \supp(p)$ so that $P$ and $Q$ are \emph{not} orthogonal.
	\end{enumerate}
\end{defn}

\begin{defn}\label{defn:"basic"_set}
	Given $r \geq 0$ and $p = \sum a_W p_W \in \partial(\mc{X},\mf{S})$ define two sets of points:
	\begin{itemize}
		\item  $\mc{A}^{rem}_r(p)$ is the set of points $q \in \partial (\mc{X},\mf{S})$ that are remote to $p$ and have $$\partial \pi_W(q) \subseteq M(r;p_w)$$  for all $W \in \supp(p)$;
		\item $\mc{A}^{non}_r(p)$ is the set of points $q \in \partial(\mc{X},\mf{S})$ that are not remote to $p$ and have  $$\partial \pi_W(q) \subseteq M(r;p_w)$$  for all $W \in \supp(p) \cap \supp(q)$.
	\end{itemize}
	Define $\mc{A}_r(p) :=  \mc{A}^{rem}_r(p) \cup \mc{A}^{non}_r(p)$.
\end{defn}

In \cite[Section 2]{HHS_Boundary}, Durham, Hagen, and Sisto describe a basis of neighborhoods for $\partial(\mc{X},\mf{S})$.  These basis sets are subsets of the $\mc{A}_r(p)$ defined by putting restrictions on the coefficients of points $q = \sum_{W \in \supp(q)}a_W q_W  \in \mc{A}_r(p)$; see \cite[Section 2]{HHS_Boundary} for details.  We therefore have Lemma \ref{lem:"basic"_sets} below. The hyperbolic case of Lemma \ref{lem:"basic"_sets} is a consequence of the fact that a hyperbolic HHS cannot have a pair of  unbounded domains that are orthogonal; see \cite[Lemma 4.1]{HHS_Boundary}.
\begin{lem}\label{lem:"basic"_sets}
	For each  $r \geq 0$ and $p \in \partial(\mc{X},\mf{S})$, the set $\mc{A}_r(p)$ contains an open set containing $p$. If $\mc{X}$ is hyperbolic, then the sets $\mc{A}_r(p)$ form a basis for the topology on $\partial(\mc{X},\mf{S})$.
\end{lem}

\subsection{The case of isolated orthogonality}\label{subsec:isolated_orthogonality_case}

For this subsection, let $(\mc{X},\mf{S})$ be an HHS with the bounded domain dichotomy, and let $S$ be the $\nest$--maximal element of $\mf{S}$. Moreover, assume that $\mf{S}$ has orthogonality isolated by $\mf{I} \subseteq \mf{S}$ and that every non-$\nest$--maximal element of $\mf{S}$ is nested into a domain in $\mf{I}$.

By Theorem \ref{thm:isolated_orthogonality_implies_rel_hyp}, this implies $\mc{X}$ is hyperbolic relative to the collection $\{\P_I:I \in \mf{I}\}$. Let $\cusp(\mc{X})$ be the cusped space obtained by attaching a combinatorial horoball to $\P_I$ for each $I \in \mf{I}$. We will prove that $\partial \cusp(\mc{X})$ is the quotient of $\partial(\mc{X},\mf{S})$ formed by collapsing the limit set of each product region $\P_I$ to a point.

To define the quotient map, we equip $\cusp(\mc{X})$ with the following HHS structure $\mf{R}$; the fact that this is an HHS structure is a direct consequence of  \cite[Theorem 3.2 and 4.2]{Russell_Relative_Hyperbolicity}.
\begin{itemize}
	\item The index set is  $\mf{R} = \{S\} \cup \mf{I}$, where $S$ is the $\nest$--maximal element of $\mf{S}$. 
	\item The $\nest$--maximal element of $\mf{R}$ is $S$ and all elements of $\mf{I}$ are transverse to each other.
	\item The hyperbolic space for $I \in \mf{I}$ is the horoball $\mc{H}(\P_I)$ and the hyperbolic space for $S$ is $\mc{C}S$.
	\item The projection maps in $\mf{R}$ are denoted $\widehat{\pi}_\ast$. For $S$, the projection $\widehat{\pi}_S \colon \cusp(\mc{X}) \to \mc{C}S$ is an extension of $\pi_S$ to the horoballs over the $\P_I$  so that $\widehat{\pi}_S(\mc{H}(\P_I)) = \rho_S^I$ and $\widehat{\pi}_S(x) = \pi_S(x)$ for $x \in \mc{X}$. For each $I \in\mf{I}$, the projection $\widehat{\pi}_I \colon \cusp(\mc X) \to \mc{H}(\P_I)$ is  defined using the gate map, $\gate_{\P_I}$, from $(\mc{X},\mf{S})$ as follows:
	\begin{itemize}
		\item if $x \in \mc{X} \subseteq \cusp(\mc{X})$, then  $\widehat{\pi}_I(x) = \gate_{\P_I}(x)$, and
		\item if $x \not \in \mc{X}$, then $x \in \mc{H}(P_J)$ for a unique $J \in \mf{I}$. In this case, $\widehat{\pi}_I(x) = \gate_{\P_I}(\P_J)$.
	\end{itemize} 

	\item The relative projections in $\mf{R}$ are denoted by $\widehat{\rho}_\ast^\ast$.
	 For each $I,J \in \mf{I}$, we have $\widehat{\rho}^{I}_S = \rho_S^I$ and $\widehat{\rho}_I^J = \gate_{\P_I}(\P_J)$. 
	 
\end{itemize}

Since $\cusp(\mc{X})$ is hyperbolic, the Gromov boundary $\partial \cusp(\mc{X})$ is naturally homeomorphic to the HHS boundary $\partial(\cusp(\mc{X}),\mf{R})$ by \cite[Lemma 4.2]{HHS_Boundary}. Hence, we will build a quotient map from $\partial (\mc{X},\mf{S})$ to $\partial(\cusp(\mc{X}),\mf{R})$. For each $I\in \mf{I}$, let $\xi_I$ denote the single element of $\partial \mc{H}(\P_I)$.

Let $x_0 \in \mc{X}$ be the basepoint for $\partial (\mc{X},\mf{S})$ and choose $x_0$ to also be the basepoint of $\partial(\cusp(\mc{X}),\mf{R})$. If $p = \sum a_Wp_W \in \partial (\mc{X},\mf{S})$, then $M(r;p_W)$ will denote the standard basis neighborhood in $\overline{\mc{C}W}$ of $p_W$, and $\mc{A}_r(p)$, $\mc{A}^{rem}_{r}(p)$, and $\mc{A}^{non}_r(p)$ will denote the sets describe in Definition \ref{defn:"basic"_set}.    For $p \in \partial (\cusp(\mc X),\mf{R})$, the support of $p$ is a single domain $W \in \mf{R}$.
Thus,  we will use $\widehat{M}(r;p)$ to denote the basis neighborhood  for $p$ in $\mc CW$, which is either $\mc{C}S$ or $\mc{H}(\P_I)$ depending on whether $W= S$ or $W\in\mf I\subset\mf R$. Similarly $\widehat{\mc{A}}_{r}(p)$,  $\widehat{\mc{A}}^{rem}_{r}(p)$, and $\widehat{\mc{A}}^{non}_{r}(p)$ will denote the sets from Definition \ref{defn:"basic"_set} applied to the HHS $(\cusp(\mc{X}),\mf{R})$.  Since $\cusp(\mc{X})$ is hyperbolic, the sets $\widehat{\mc{A}}_r(p)$  form a basis for the topology on $\partial(\cusp(\mc{X}),\mf{R})$ by Lemma \ref{lem:"basic"_sets}.

We say a subset $\mf{U} \subseteq \mf{S}$ is \emph{entirely nested} into a domain $W \in \mf{S}$ if $V \nest W$ for each $V \in \mf{U}$. Because every domain of $\mf{S}- \{S\}$ is nested into an element of $\mf{I}$ and $\mf I$ isolates orthogonality, for each $p \in \partial (\mc{X},\mf{S})$ either $\supp(p) =\{S\}$ or $\supp(p)$ is entirely nested in some $I \in \mf{I}$.

\begin{prop}\label{prop:isolated_case_Phi_is_cont}
	The map  $\Phi \colon \partial (\mc{X},\mf{S}) \to \partial (\cusp(\mc X),\mf{R})$ given by 
	\[ \Phi(p) = \begin{cases} p & \text{if }\supp(p) = \{S\} \\ \xi_I & \text{if } \supp(p) \text{  is entirely nested in } I \in \mf{I}
		
	\end{cases}\]
	is continuous and surjective.  Moreover, if $\iota \colon \mc{X} \to \cusp(\mc{X})$ is the inclusion map and $(x_n)$ is a sequence of points in $\mc{X}$ that converges to $p$, then $(\iota(x_n))$ converges to $\Phi(p)$.
\end{prop}

\begin{proof}
	We first prove two claims that describe the images of $\mc{A}^{rem}_r(p)$ and $\mc{A}^{non}_r(p)$ under $\Phi$.  The claims are divided based on the support of $p$, which must either be equal to $\{S\}$ (Claim \ref{claim:supp(p)=S}) or entirely nested in some $I\in\mf I$ (Claim \ref{claim:supp(p)inI}). Let $E$ be the hierarchy constant for $\mf{S}$ and $\mf{R}$.
	
	\begin{claim}\label{claim:supp(p)=S}
	Suppose  $p \in \partial(\mc{X},\mf{S})$ with $\supp(p) = \{S\}$. For all $r \geq 0$, there exist $r' \geq 0$ so that     $\Phi(\mc{A}_{r'}(p)) \subseteq  \widehat{\mc{A}}_{r}(p)$.
	\end{claim}
	
	\begin{proof}
		Since $\supp(p) =\{S\}$, we have  $\widehat{M}(r;p) = M(r;p)$ for all $r \geq 0$. For each $r\ge 0 $, there exists $r'\geq r$ so that whenever $x \in M(r';p)\cap \mc{C}S $, then $\mc{N}_{3E}(x) \subseteq M(r;p)$. Such an $r'$ depends only on $r$ and $E$.
		
		Let $q \in\mc{A}_{r'}(p)$.
		If $\supp(q) = \{S\}$, then $q \in M(r';p) \subseteq M(r;p)$. However $\widehat{M}(r;p) = M(r;p)$, and so  $q \in \widehat{\mc{A}}_{r}(p)$. 
		If instead $\supp(q)$ is entirely nested in $I \in \mf{I}$, then for each $V \in \supp(q)$, we have $V \nest I \propnest S$. The consistency axiom in $\mf{S}$ ensures that each such $\rho_S^V$ is contained in $\mc{N}_{2E}(\rho_S^I)$. Since $\partial \pi_S(q)$ is the union of the $\rho_S^V$ over all $V \in \supp(q)$, we have $\rho_S^I \subseteq  \mc{N}_{2E}(\partial \pi_S(q))$. Since $q \in  \mc{A}_{r}(p)$ and $\supp(p) = \{S\}$, the set $\partial\pi_S(q)$ must be contained in $M(r';p)$. Thus $\mc{N}_{2E}(\partial \pi_S(q)) \subseteq M(r;p)$. Since $ M(r;p) = \widehat{M}(r;p)$, we have $$\rho_S^I = \widehat{\rho}_S^I \subseteq \widehat{M}(r;p),$$ and thus $\Phi(q) = \xi_I \in \widehat{\mc{A}}_r(p)$. 
	\end{proof}

	\begin{claim}\label{claim:supp(p)inI}
	Suppose $p \in \partial(\mc{X},\mf{S})$ with $\supp(p)$ entirely nested into $I \in \mf{I}$.  For all $r \geq 0$,   there exists $r' \geq0$, so  $\Phi(\mc{A}_{r'}(p)) \subseteq  \widehat{\mc{A}}_{r}(\xi_I) = \widehat{\mc{A}}_{r}(\Phi(p))$.
	\end{claim}
	
	\begin{proof}
	Let $q\in \mc A_{r'}(p)$ for some $r'>0$.   The proof  is divided into three cases. In each case, we will show that  if $r'$ is sufficiently large, then  $\Phi(q)\in \widehat{\mc A}_r(\xi_I)$. \\
	
	\noindent \textit{Case 1: $q\in \mc A^{non}_{r'}(p)$.} 
	Because of isolated orthogonality, if $q$ is not remote to $p$, then either $\supp(p) = \supp(q) = \{S\}$ or there is a single $I \in \mf{I}$ so that $\supp(p)$ and $\supp(q)$ are both entirely nested into $I$. Since we are working under the assumption that $\supp(p)$ is entirely nested in $I\in \mf I$, the same must be true of $\supp(q)$, and we conclude that $\Phi(q) = \xi_I\in \widehat{\mc A}_r(\xi_I)$. \\

	\noindent \textit{Case 2: $q\in \mc A^{rem}_{r'}(p)$ and $\supp(q)$ is entirely nested in $I$.}  In this case $\Phi(q)=\xi_I\in \widehat{\mc A}_r(\xi_I)$.
	\\

	\noindent \textit{Case 3: $q\in \mc A^{rem}_{r'}(p)$ and $\supp(q)$ is not entirely nested in $I$.} 
	Each $\P_I$ is uniformly a quasi-geodesic space by virtue of being uniformly hierarchically quasiconvex and Proposition \ref{prop:HQC_and_hp}.  Let $f\colon [0,\infty)\to [0,\infty)$ be the function from Lemma~\ref{lem:horoball_base_distance_to_boundary} for $\mc{H}(\P_I)$. Fix $W\in \supp(p)$.  By the assumptions of Claim \ref{claim:supp(p)inI}, $W\nest I$.
	Since $q$ is in $\mc A^{rem}_{r'}(p)$, we have $\partial \pi_W(q)\subseteq M(r';p_W)$.	Under the assumptions of this case, either the support of $q$ is  entirely nested into some $J\in\mf I-\{I\}$ or $\supp(q)=\{S\}$.  We will deal with each possibility in a separate subcase.
	
	In both subcases, the strategy of the proof is to show that $\partial\widehat{\pi}_I(\Phi(q))\subseteq \widehat{M}(r;\xi_I)$.  To do this, let $y$ be a point in $\mc X$ so that $\widehat{\pi}_I(y)\in \partial\widehat{\pi}_I(\Phi(q))$.   If we can show that $d_W(\pi_W(x_0),\pi_W(y))$ is sufficiently large, then since the maps $\pi_W$ are coarsely Lipschitz,  we can conclude that $$d_{\P_I}(\gate_{\P_I}(x_0),\gate_{\P_I}(y))>f(r).$$  By Lemma~\ref{lem:horoball_base_distance_to_boundary}, this would show that $\widehat{\pi}_I(y)\in \widehat{M}(r;\xi_I)$, as desired.
	
	\textit{Case 3a: $\supp(q)$ is  entirely nested into some $J\in\mf I-\{I\}$.}  
	In this case, $\Phi(q)=\xi_J$, and $\xi_J$ is remote to $\xi_I$ because $\supp(\xi_J) = \{J\}$, $\supp(\xi_I) = \{I\}$, and $J \trans I$.  By definition, $$\partial\widehat{\pi}_I(\xi_J)=\widehat{\rho}^J_I=\gate_{\P_I}(\P_J).$$ Since $\supp(q)$ is entirely nested in $J$, the projection $\rho^J_W$  is coarsely equal to $\partial\pi_W(q)$, which is contained in $M(r';p_W)$ by assumption.  Moreover, $\pi_W(\gate_{\P_I}(\P_J))$ is uniformly close to $\rho_W^J \subseteq M(r';p_W)$. 
 Therefore, letting $y$ be any point in $\P_J$ and choosing $r'$  large enough, we can ensure that $$d_W(\pi_W(x_0),\pi_W(y))\geq d_W(\pi_W(x_0),M(r';p_W))$$ is large enough so that $d_{\P_I}(\gate_{\P_I}(x_0),\gate_{\P_I}(y))>f(r)$.  Therefore, as described above, 
	\[
	\partial\widehat{\pi}_I(\Phi(q))=\partial\widehat{\pi}_I(\xi_J)=\gate_{\P_I}(\P_J)\subseteq \widehat{M}(r;\xi_I),\]
	and we conclude that $\Phi(q)\in \widehat{\mc A}_r(\xi_I)$, as desired.
		
	\textit{Case 3b: $\supp(q)=\{S\}$.}  
	In this case, $\Phi(q)=q$, and $q$ is remote to $\xi_I$ because $S$ is not orthogonal to $I$.  For any $U\neq S$, we  have $\widehat{\rho}^U_S=\rho^U_S$ and  $\widehat{\rho}^S_U(\pi_S(y)) \subseteq \partial\widehat{\pi}_U(q)$ for any $y \in \mc{X}$ where $\pi_S(y)$ lies on a quasigeodesic ray from $\widehat{\rho}^U_S$ to $q\in \partial CS$ that is sufficiently far from $\widehat{\rho}^U_S$.

	Since $W\nest I\nest S$, the upward projections $\rho^W_S$ and $\rho^I_S$ are coarsely equal.  Thus there exists $y\in \mc X$ such that $\pi_S(y)$ lies on a quasigeodesic from $\rho^I_S$ to $q$ and is sufficiently far from both $\rho^W_S$ and $\rho^I_S$ so that $\rho^S_W(\pi_S(y))\subseteq \partial \pi_W(q)$ and $\rho^S_I(\pi_S(y))\subseteq \partial\widehat{\pi}_I(q)$.  In particular, the first inclusion  implies that $\rho^S_W(\pi_S(y))\subseteq M(r';p_W)$.  
	
	By Lemma \ref{lem:downward_rhos}, $\pi_W(y)$ and $\rho^S_W(\pi_S(y))$ are uniformly coarsely equal. Since $W\nest I$, the projections $\pi_W(y)$ and $\pi_W(\gate_{\P_I}(y))$ are also uniformly coarsely equal.  Since $$\rho^S_W(\pi_S(y))\subseteq \partial \pi_W(q)\subseteq M(r';p_W),$$ there is some $c>0$ depending only on $E$ such that $\pi_W(\gate_{\P_I}(y))\subseteq M(r'-c;p_W)$.

	As in the previous subcase,  choosing $r'$ sufficiently large ensures that $d_{\P_I}(\gate_{\P_I}(x_0),\gate_{\P_I}(y))$ is greater than $f(r)$. Therefore, 
	\[
	\gate_{\P_I}(y)\in \widehat{M}(r; \xi_I).
	\]
By our choice of $y$, we have $\gate_{\P_I}(y)\subseteq \partial\widehat{\pi}_I(q)$, which has uniformly bounded diameter.  Thus by making $r'$ even larger, we can ensure that
\[
\partial\widehat{\pi}_I(q)\subseteq \widehat{M}(r; \xi_I).
\]
We conclude that $\Phi(q)=q\in \widehat{\mc A}_r(\xi_I)$, completing the proof of the claim.	
	\end{proof}
	
	The proof that $\Phi$ is continuous is now a direct application of the above claims. Let $O$ be an open subset  of $\partial(\cusp( \mc{X} ),\mf{R})$ and $p \in \Phi^{-1}(O)$.  Since the $\widehat{\mc{A}}_r(\cdot)$ sets form a basis for the topology on $\partial(\cusp(\mc{X}),\mf{R})$, there exists $r \geq 0$ so that $\widehat{\mc{A}}_r(\Phi(p)) \subseteq O$. By Claims \ref{claim:supp(p)=S} and \ref{claim:supp(p)inI}, there then exists $r' \geq 0$ so that $\Phi(\mc{A}_{r'}(p) ) \subseteq \widehat{\mc{A}}_r(\Phi(p)) \subseteq O$. This shows $\Phi^{-1}(O)$ is open, as  $\mc{A}_{r'}(p)$ contains an open set containing $p$ by Lemma \ref{lem:"basic"_sets}.

	Lastly, we prove the moreover claim of the proposition. Let $(x_n)$ be a sequence of points in $\mc{X}$ that converges to the boundary point $p \in \partial(\mc{X},\mf{S})$.
	
	Suppose first that $\supp(p) = \{S\}$. Lemma \ref{lem:convergence_of_interior_points} implies that for each $r \geq 0$, we have $\pi_S(x_n) \subseteq M(r;p)$ for all but finitely many $n$.   Since $\Phi(p) = p$, $M(r;p) = \widehat{M}(r;p)$, and $\widehat{\pi}_S(\iota(x_n)) = \pi_S(x_n)$, we have $\widehat{\pi}_S(\iota(x_n)) \subseteq \widehat{M}(r;p)$ for all but finitely many $n$. This shows that $(\iota(x_n))$ converges to $\Phi(p) = p$ in $\cusp(\mc{X}) \cup \partial (\cusp(\mc{X}),\mf{R})$ because the sets $\widehat{\mc{A}}_r(\cdot)$ form a basis for the topology on $\partial(\cusp(\mc{X}))$.
	
	Now suppose $\supp(p)$ is totally nested into $I \in \mf{I}$.  For each $V \in \supp(p)$, the distance $d_{V}(x_0,x_n)$ goes to infinity as $n \to \infty$. Since any such $V$ is nested into $I$, the  coarse Lipschitzness of the projection maps in $\mf{S}$ says $d_{\P_I}(\gate_{\P_I}(x_0), \gate_{\P_I}(x_n))$ also goes to infinity as $n \to \infty$. Hence by Lemma \ref{lem:horoball_base_distance_to_boundary}, for all but a finite number of $n$, we have $\gate_{\P_I}(x_n) \in \widehat{M}(r;\xi_I)$  for any $r$. Since $\widehat{\pi}_I(\iota(x_n)) = \gate_{\P_I}(x_n)$ and $\xi_I = \Phi(p)$, this shows $(\iota(x_n))$ converges to $\Phi(p)$  in $\cusp(\mc{X}) \cup \partial (\cusp(\mc{X}),\mf{R})$.
\end{proof}

\subsection{Proof of Theorem \ref{thm:bowditch_boundary}}\label{subsec:proof_of_bowditch_boundary}

Let $(G,\mf{S})$ be a $G$--HHS that is hyperbolic relative to the finite collection of subgroups $\mc{P}$. Let $\mf{T}$ be the maximization of $\mf{S}$,  and let $\mf{H}$ be the $G$--HHS structure for $G$ that comes from adding the cosets of the peripheral subgroups to  $\mf{T}$  as described in Construction \ref{con:HHG+hyp_embedded} and Theorem \ref{thm:adding_in_hyperbolically_embedded_subgroups}. By Corollary \ref{cor:rel_hyp_structure}, $\mf{H}$ has orthogonality isolated by $\mf{Q}$, the set of domains indexing the cosets of the peripheral subgroups. Moreover, every non-$\nest$--maximal element of $\mf{H}$ is nested into an element of $\mf{Q}$.

As described in Section \ref{subsec:isolated_orthogonality_case},  there is an HHS structure for $\cusp(G,\mc{P})$ with index set $\mf{R} = \{S\} \cup \mf{Q}$ and a continuous surjection  of HHS boundaries $\Phi \colon \partial (G,\mf{H}) \to \partial(\cusp(G,\mc{P}),\mf{R})$. 

Since the Cayley graph of $G$ is a proper metric space, $\cusp(G,\mc{P})$ is also proper. In particular, $ \partial (G,\mf{H})$ and $\partial(\cusp(G,\mc{P}),\mf{R})$ are both compact, Hausdorff spaces. Hence,  every surjective continuous map between these HHS boundaries is a quotient map.  In particular,  Proposition~\ref{prop:isolated_case_Phi_is_cont} shows that $\Phi$ is a quotient map. By construction, $\Phi(p) = \Phi(q)$ if either $p =q$ or $\supp(p)$ and $\supp(q)$ are both totally nested into a domain $ Q\in \mf{Q}$. By Lemma  \ref{lem:limit_set_support}, a point in $\partial(G,\mf{H})$ has support totally nested into $Q \in \mf{Q}$ if and only if that point lies in the limit set of the coset  $\coset{Q}$ indexed by $Q$.  This implies $\Phi(p) = \Phi(q)$ for distinct $p$ and $q$ precisely when $p$ and $q$ are in the limit set of the same coset of a group in $\mc{P}$. 

The homeomorphisms $\partial(G,\mf{S}) \to \partial(G,\mf{T})$ and $\partial(G,\mf{T}) \to \partial(G,\mf{H})$ from Theorem \ref{thm:maximization_invariance}\eqref{invar:boundary} and Corollary \ref{cor:adding_hyp_embedded+boundary} pointwise preserve the limit set of each coset of the peripheral subgroups because they are continuous extensions of the identity. By composing these maps and then following with the map $\Phi$, we have the desired quotient map $\Psi \colon\partial(G,\mf{S}) \to  \partial(\cusp(G,\mc{P}),\mf{R})$. Since $\cusp(G,\mc{P})$ is hyperbolic, the Bowditch boundary $\partial \cusp(G,\mc{P})$ is homeomorphic to $\partial(\cusp(G,\mc{P}), \mf{R})$.

Since the homeomorphisms from Theorem \ref{thm:maximization_invariance}\eqref{invar:boundary} and Corollary \ref{cor:adding_hyp_embedded+boundary} are continuous extensions of the identity map on $G$, the moreover clause of Proposition \ref{prop:isolated_case_Phi_is_cont} says that when a sequence of point in $G$ converges to a boundary point  $p \in \partial(G,\mf{S})$,  the inclusion of that sequence into $\cusp(G,\mc{P})$ will converge to the image of $p$ in the quotient of the boundary.  Hence, we have completed the proof of Theorem \ref{thm:bowditch_boundary}.

\section{The boundary of thick $G$--HHSs}\label{sec:thick_case}

In this section, we examine the connection between the simplicial
structure on the HHS boundary and a geometric obstruction to relative
hyperbolicity called thickness.  We start with some background on
thick metric spaces in Section \ref{subsec:thick_spaces}. We then use the HHS boundary to characterize when $G$--HHSs, and their
hierarchically quasiconvex subgroups, are thick of order 0 in Section
\ref{subsec:wide}.  Finally, we give a 
characterization of when a $G$--HHS is thick of order 1 in Section \ref{subsec:thick_order_1}.

\subsection{Thick metric space} \label{subsec:thick_spaces}

 Behrstock, Dru\c{t}u, and Mosher introduced the notion of thickness as
 a geometric obstruction to a space being relatively hyperbolic
 \cite{BDM_Thickness}.  Thickness is defined inductively with the
 following spaces forming the base level of the induction.

\begin{defn}[Wide metric space]\label{defn:wide}
	A quasi-geodesic metric space $X$ is \emph{wide} if none of its asymptotic cones
	have cut points. A subset $Y$ of $X$
	is \emph{wide} if the restriction of the metric of $X$ to $Y$
	makes $Y$ a wide metric space.  A finitely generated group is wide
	if the word metric with respect to a finite
	generating set is wide.
\end{defn}

A basic example of a wide space is one which is quasi-isometric to a product of two infinite diameter, 
quasi-geodesic metric spaces. A more subtle example is provided by 
Baumslag--Solitar groups. 

To every thick space there is an associated non-negative integer, which is 
its order of thickness. Wide spaces are the spaces that are thick of order 0.  Higher orders of thickness are
obtained by inductively chaining together thick spaces of lower order.
In the present paper, we only consider spaces that are thick of order 0 or
1; see \cite{BDM_Thickness} for further details about higher orders
of thickness.

\begin{defn}[Thick of order 1]\label{defn:thick_order_1}
	A quasi-geodesic metric space $X$ is \emph{thick of order 0} if it is wide.
	A quasi-geodesic metric space $X$ is \emph{thick of order $1$} if it is not wide and there exists a constant $C\geq 0$ and a collection of wide subsets $\{P_\alpha\}_{\alpha \in I}$ so that:
	\begin{enumerate}
		\item (Coarse Cover) The space $X$ is contained in the
		$C$--neighborhood of $\bigcup_{\alpha \in I} P_\alpha$.  
		\item (Thick Chains) For any $P_\alpha$ and $P_{\alpha'}$ 
		that both intersect 
		$\mc{N}_{3C}(x)$
		for some $x\in X$,
		there
		exists a sequence $$P_\alpha = P_0, P_1, \dots, P_k =
		P_{\alpha'}$$ such that  $\mc{N}_C(P_i) \cap \mc{N}_C(P_{i+1})$  
		has infinite diameter for all $0\leq i \leq k-1$. We call the sequence  $P_0, P_1, \dots, P_k$ a \emph{thick chain from $P_\alpha$ to $P_{\alpha'}$}.
	\end{enumerate}
	When $X$ is a  finitely generated group $G$ equipped with a word metric and the collection of subsets $\{P_\alpha\}$ is the set of left cosets of a finite number of undistorted subgroups $H_1,\dots,H_n$, then we say $G$ is \emph{thick of order 1 relative to} $H_1,\dots,H_n$
\end{defn}

While the above definition of thickness is sufficient to obstruct
relative hyperbolicity, the definition below of strongly thick was
introduced by Behrstock and Dru\c{t}u to yield lower
bounds on divergence from thickness; see \cite{BehrstockDrutu:thick2}.

For the remainder of the section, we say a
subset $Y$ of a metric space $X$ is \emph{quasiconvex} if there is
$\lambda \geq 1$ and $\varepsilon \geq0$ so that for every pair of
points $x,y \in Y$ there is a $(\lambda,\varepsilon)$--quasi-geodesic
$\gamma$ from $x$ to $y$ with $\gamma \subseteq
\mc{N}_\varepsilon(Y)$. 
This notion of quasiconvexity is preserved by quasi-isometries of the
space.  The original setting in which quasiconvexity was defined is
for hyperbolic spaces.  There,   since quasi-geodesics are
uniformly close to geodesics, it is equivalent to use geodesics rather
than quasi-geodesics when defining quasiconvexity, and indeed, this is the
standard way in which quasiconvexity is defined.  Outside of the
hyperbolic setting, using geodesics one would not obtain a notion which 
is preserved by quasi-isometries, which is why the definition using 
quasi-geodesics is more natural in the study of coarse geometry and thus what we use in this section.

\begin{defn}[Strongly thick of order 1]\label{defn:strong_thickness}
	Let $X$ be a metric space that is thick of order 1 with respect to the constant $C\geq 0$ and the collection of subsets 
	$\{P_\alpha\}_{\alpha \in I}$. We say $X$ is \emph{strongly thick of order 1} if  each $P_\alpha$ is uniformly quasiconvex and there exists a number $\tau \geq 0$ so that if $P_\alpha$ and $P_{\alpha'}$ intersect  $\mc{N}_{3C}(x)$ for some $x\in X$, then any  thick chain  $P_\alpha = P_0, P_1, \dots, P_k =
	P_{\alpha'}$ has  $k \leq \tau$ and each coarse intersection $\mc{N}_C(P_i) \cap \mc{N}_C(P_{i+1})$  is $\tau$--coarsely connected and intersects $\mc{N}_\tau(x)$.
\end{defn}

The next result gives some fairly general conditions for deducing
strong thickness from thickness.  The special case where the
collection $\mathcal{P}$ is the collection of left cosets of a finite
set of quasiconvex subgroups $\mc{H}$ follows immediately from
\cite[Proposition~4.4]{BehrstockDrutu:thick2}.

\begin{prop}\label{prop:thickimpliesstronglythick}
	Let $X$ be thick of order 1 with respect 
	to a collection $\mathcal{P}$. Let $G$ be a finitely generated 
	group acting coboundedly on $X$ by isometries   so that: 
	\begin{itemize}
		
		\item the elements of $\mathcal{P}$ are each uniformly quasiconvex;
		
		\item  the infinite diameter coarse intersection of any two elements of 
		$\mathcal{P}$  in the Thick Chains condition is uniformly coarsely 
		connected; and
		
		\item  $\mathcal{P}$ is $G$--invariant with respect to the action of $G$ on $X$. 
		
	\end{itemize}
	
\noindent Additionally, assume that either one of the following conditions are
	satisfied:
		\begin{enumerate}
		\item \label{item:finite_multiplicity} every closed ball in $X$ intersects a finite number of elements of $\mc{P}$.
		\item \label{item:finte_pair_orbits} the induced action of $G$ on $\mc{P} \times \mc{P}$ has finitely many orbits.
	\end{enumerate}
	
	\noindent Then $X$ is strongly thick of order 1 with respect to
	$\mathcal{P}$. 
\end{prop}

\begin{proof} 
	Let $C\geq0$ be the thickness constant, and let $B \geq0$ be the diameter of the quotient $X /G$.
	
		Two of the requirements of strong thickness hold by our bulleted assumptions:  uniform quasiconvexity of the subsets in $\mathcal{P}$ and uniformly coarse connectedness of the coarse intersections of successive elements of any thick chain. What remains to be shown is that there exists a uniform $\tau \geq 0$ so that for  any two elements $P,P' \in \mc{P}$ that intersect $\mc{N}_{3C}(x)$ for some $x \in X$,  there exists a thick chain $P=P_0,P_1\dots, P_k=P'$ with $k\leq \tau$ and where $\mc{N}_C(P_i) \cap \mc{N}_{C}(P_{i+1})$ intersects $\mc{N}_\tau(x)$ for each $i \in \{0,\dots, k-1\}$. For this we will need one of the two numbered hypotheses.

	Suppose first that we assume hypothesis 
	(\ref{item:finite_multiplicity}): every closed ball in $X$ intersects a finite number of elements of $\mc{P}$. Fix $x_0 \in X$ and let $R_1,\dots, R_m$ be all of the elements of $\mc{P}$ that intersect $\mc{N}_{3C+2B}(x_0)$. Since $X$ is thick of order 1, for each pair $R_i$, $R_j$ there exists a thick chain of subsets of $\mc{P}$ from $R_i$ to $R_j$. For each $i,j$ pair, 	fix one such chain, $\mf{C}_{i,j}$. Let $\tau \geq k$ be large enough so that   $\tau \geq |\mf{C}_{i,j}|$  and  the intersections of the $C$--neighborhood of consecutive elements of the chain $\mf{C}_{i,j}$ intersect $\mc{N}_\tau(x_0)$ for each $i,j$ pair.

	Now, let $P,P'$ be  elements of $\mc{P}$ that intersect $\mc{N}_{3C}(x)$ for some $x \in X$. There exists $g\in G$ so that $gP$ and $gP'$ intersect $\mc{N}_{3C +2B}(x_0)$. Hence $gP = R_i$ and $gP' = R_j$ for some $i,j$. Thus, $g^{-1} \mathfrak{C}_{i,j}$ is the desired thick chain from $P$ to $P'$.
	
	Now assume instead hypothesis (\ref{item:finte_pair_orbits}): the
	action of $G$ on $\mc{P} \times \mc{P}$ has finitely many orbits.
	Let $\{(R_1,Q_1),\dots,(R_m,Q_m)\}$ be representatives of the
	finitely many $G$--orbits in $\mc{P} \times \mc{P}$.  By the 
	equivariance in the third bullet point, without loss
	of generality, we can assume each $R_i$ is within $B$ of a fixed
	point $x_0 \in X$.  For each $i \in \{1,\dots, m\}$, there is a
	thick chain of elements of $\mc{P}$ from $R_i$ to $Q_i$.  For each
	$i$, fix one such thick chain $\mf{C}_i$.  Let $\tau$ be large
	enough so that $\tau \geq |\mf{C}_{i}|$ and the intersections of
	the $C$--neighborhood of consecutive elements of the chain
	$\mf{C}_{i}$ intersect $\mc{N}_\tau(x_0)$ for each $i$.
	
	Now let $P,P'$ be  elements of $\mc{P}$ that intersect $\mc{N}_{3C}(x)$ for some $x \in X$. There is $g \in G$ so that $gP = R_i$ and $gP' = Q_i$ for some $i$. Hence, $g^{-1}\mf{C}_i$ is the desired thick chain.
\end{proof}

\subsection{Wide hierarchically quasiconvex subgroups} \label{subsec:wide}
In this subsection, we characterize wide hierarchically quasiconvex 
subgroups as those whose limit sets are non-trivial joins. 
\begin{thm}\label{thm:wide_iff_join}
	Let $(G,\mf{S})$ be a $G$--HHS and let $H<G$ be an infinite, hierarchically quasiconvex subgroup. The group $H$ is wide if and only if the limit set $\Lambda(H)$ in $\partial_\Delta(G,\mf{S})$ is a non-trivial join. In particular, $G$ is wide if and only if $\partial_\Delta(G,\mf{S})$ is a join.
\end{thm}

For the entire $G$--HHS, Theorem \ref{thm:wide_iff_join} is direct
consequence of the Rank Rigidity Theorem \cite[Theorem
9.13]{HHS_Boundary} (see also \cite[Corollary
4.7]{PS_Unbounded_domains}).  For subgroups, however, there are some
subtleties which need to be addressed before results
from the literature can be applied.  These arise from the fact that
unlike the entire group, the subgroup $H$ might have projections that
are bounded but arbitrarily large.

The starting point in our proof of Theorem \ref{thm:wide_iff_join} is
the theorem of Petyt and Spriano below, which applies to
hierarchically quasiconvex subgroups as they are always finitely
generated (Lemma \ref{lem:HQC_undistorted}).  In the sequel, we will
use $\mf{S}_H^\infty$ to denote the set of domains $\{V \in \mf{S}:
\diam(\pi_V(H)) = \infty\}$ for any subgroup of $H$ of a $G$--HHS
$(G,\mf{S})$. The set of
	domains $\{W_1,\dots, W_n\}$ obtained in Theorem~\ref{thm:eyries} 
	are called the \emph{eyries} for $H$.

\begin{thm}[{Special case of \cite[Theorem 5.1]{PS_Unbounded_domains}}]\label{thm:eyries}
	Let $(G,\mf{S})$ be a $G$--HHS. For every infinite, finitely
	generated subgroup $H<G$, there exists a non-empty, pairwise
	orthogonal set of domains $\{W_1,\dots,W_n\} \subseteq
	\mf{S}_H^\infty$ so that for all $V \in \mf{S}_H^\infty$ we have
	$V \nest W_i$ for some $i \in \{1,\dots,n\}$.  
\end{thm}

Since the vertices of the limit set of $H$ are supported on domains in $\mf{S}_H^\infty$, the limit set of $H$ is a join if and only if $H$ has multiple eyries. The challenge, then,  is to show that having multiple eyries is equivalent to the hierarchically quasiconvex subgroup being wide. The key technical step is to establish that there is a uniform bound for the diameter of the projection of $H$ onto any domain not nested into an eyrie. 

\begin{lem}\label{lem:bounded_proj_outside_eyries}
	Let $(G,\mf{S})$ be a $G$--HHS and $H <G$ be an infinite, hierarchically quasiconvex subgroup. There exists $D \geq 0$, depending on $H$, so that $\diam(\pi_V(H)) \leq D$ whenever $V \in \mf{S}$  is not nested into an eyrie for $H$.
\end{lem}

Our proof of Lemma \ref{lem:bounded_proj_outside_eyries} requires three tools from the literature. The first is a basic technique in the theory of hierarchically hyperbolic spaces that allows one to convert many large projections into a bigger projection higher up the $\nest$--lattice.

\begin{lem}[{Passing-up lemma, \cite[Lemma 2.5]{BHS_HHSII}}]\label{lem:passing-up} 
	Let $(\mc{X},\mf{S})$ be a hierarchically hyperbolic space with constant $E$. For every $C \geq 0$, there is a positive integer $p=p(C)$ so that for all 
	$x,y \in \mc{X}$, if there exist $p$ domains $\{U_1,\dots, U_p\} \subseteq \mf{S}$ with $d_{U_i}(x,y) >E$ for each $U_i$, then there is a domain $W \in \mf{S}$ so that $d_W(x,y) >C$ and there is some $U_i$ properly nested into $W$.   
\end{lem}

The second result combines two technical lemmas from the work of Petyt and Spriano. We state the version of their work that we apply and describe how to translate from the statements in \cite{PS_Unbounded_domains} to the statement below.

\begin{lem}[{Special case of \cite[Lemmas 3.4 and 3.5]{PS_Unbounded_domains}}]\label{lem:pass-up_to_infinite}
	Let $(G,\mf{S})$ be a $G$--HHS with constant $E$ and $H<G$ a
	subgroup.  Suppose there exist domains $V_0,V_1 \in\mf{S}$ and
	$\varepsilon \geq 1$ so that: 
	\begin{itemize}
		\item  $V_1 \trans V_2$;
		\item $\pi_{V_i}(H)$ is $\varepsilon$--coarsely connected for $i=0,1$;
		\item  $\diam(\pi_{V_i}(H)) > 10^{E+1}(\varepsilon+ d_{V_i}(\rho_{V_i}^{V_j}, H))$ for $(i,j) =(0,1)$ or $(1,0)$; and 
		\item  $\diam(\pi_{V_0}(H)) >10E$.
	\end{itemize}
	Then there exist a sequence of domains $(U_i)_{i=1}^\infty$ and a sequence of points $(z_i)_{i=0}^\infty\subseteq H$ so that each $U_i$ is in the $H$--orbit of either $V_0$ or $V_1$ and $d_{U_j}(z_0, z_i) >8E$ for all $j \leq i$.
\end{lem}

\begin{proof}
	First we remark that while Lemmas 3.4 and 3.5 of \cite{PS_Unbounded_domains} are stated for HHGs and not $G$--HHSs, their proofs do not use the finiteness of orbits of domains. Hence the conclusions of both lemmas hold equally well for $G$--HHSs. The first three bullet points ensure that each of $V_0$ and $V_1$ satisfy hypothesis (b) of \cite[Lemma 3.4]{PS_Unbounded_domains} with respect to the other. The fourth bullet point ensures that there exist $z_0 \in H$ so that $d_{V_0}(z_0, \rho_{V_0}^{V_1}) > 2E$. Together, this implies $(H,V_0,V_1)$ satisfies the hypothesis of \cite[Lemma 3.5]{PS_Unbounded_domains} required to produce the desired sequences of domains and elements of $H$.
\end{proof}

The last tool implies that 
large projections for a hierarchically quasiconvex subset implies 
close proximity to the corresponding product region. This is a straightforward consequence of \cite[Proposition 4.24]{RST_Quasiconvexity} and Proposition \ref{prop:HQC_and_hp}.

	\begin{lem}\label{lem:big_proj_implies_close}
	Let $\mc{Y}$ be a $k$--hierarchically quasiconvex of an HHS
	$(\mc{X},\mf{S})$.  There exists $\nu \geq 0$, depending only on
	$k$ and the hierarchy constant of $(\mc{X},\mf{S})$, so that for
	any domain $V\in \mf{S}$, if $\diam(\pi_V(\mc{Y})) \geq \nu$, then $d_\mc{X}(\mc{Y},\P_V) \leq \nu$.
\end{lem}

We now prove Lemma \ref{lem:bounded_proj_outside_eyries}.

\begin{proof}[Proof of Lemma \ref{lem:bounded_proj_outside_eyries}]
Let $H$ be a hierarchically quasiconvex subgroup of the $G$--HHS $(G,\mf{S})$. Let $E$ be the hierarchy constant for $(G,\mf{S})$. We want to show that there exist $D \geq 0$ so that for all $V \in\mf{S}$, if $V$ is not nested into an eyrie for $H$, then $\diam(\pi_V(H)) \leq D$.

	For the purposes of contradiction, assume that there exists a sequence of domains $(V_i)$ so that:
		\begin{enumerate}[(I)]
			\item\label{item:not_eyrie} no $V_i$ is  contained in an eyrie for $H$ 
			 (and hence  $\diam(\pi_{V_i}(H))<\infty$); and
			\item\label{item:growing_diam} $\diam(\pi_{V_i}(H)) \to 
			\infty$ as $i \to \infty$.
			\setcounter{list}{\value{enumi}}
		\end{enumerate}
	
	We then define the \emph{level}, $\ell(W)$, of the domain $W \in\mf{S}$
	to be the maximal length of a descending $\nest$--chain in
	$\mf{S}$ terminating at $W$ (i.e., the $\nest$--maximal domain has
	level 1, the domains one step down have level 2 and so forth).
	Because the length of $\nest$--chains are bounded by $E$, there
	must exist some level $\ell_0 \in \mathbb{N}$ where a sequence of
	domains satisfying \eqref{item:not_eyrie} and
	\eqref{item:growing_diam} exists, but no such sequence exists for
	any level strictly less than $\ell_0$.  In particular, there is a
	number $C\geq 0$ and a sequence of domains $(V_i)$ that satisfy
	\eqref{item:not_eyrie}, \eqref{item:growing_diam}, and also:

	\begin{enumerate}[(I)]
	\setcounter{enumi}{\value{list}}
			\item\label{item:nest_max} if  $W\in \mf{S}$  satisfies $V_i \propnest W$ for some $i$, then $\diam(\pi_W(H)) <C-1$.	
	\end{enumerate}

	Let $(V_i)$ and $C \geq 0$ be the sequence and constant 
	constructed above. The remainder of the proof by contradiction 
	proceeds as follows. First we use the sequence $(V_i)$ to produce 
	a pair of domains where we can apply Lemma 
	\ref{lem:pass-up_to_infinite}. We then use the Passing-up Lemma 
	(Lemma \ref{lem:passing-up}) to produce a domain $W$ that properly contains one of the $V_i$ and has $\diam(\pi_W(H)) >C$, contradicting \eqref{item:nest_max}.
	
	Let $\nu\geq 0$ be the constant from Lemma \ref{lem:big_proj_implies_close} for the hierarchically quasiconvex subset $H$  in the HHS $(G,\mf{S})$. Since $H$ is hierarchically quasiconvex, it is also finitely generated (Lemma \ref{lem:HQC_undistorted}). As the projection maps $\pi_W$ are $(E,E)$--coarsely Lipschitz,  there exists  $\varepsilon\geq0$ so that $\pi_W(H)$ is $\varepsilon$--coarsely connected for all $W \in\mf{S}$.
	
	By passing to a subsequence,  we can assume that for each $V_i$ both $$d_G(H, \P_{V_i}) \leq \nu  \text{ and }  \diam(\pi_{V_i}(H)) > 10^{E+1}(\varepsilon + E\nu +2E).$$ 
	Since every infinite set of domains contains a pair of transverse elements \cite[Lemma 2.2]{BHS_HHSII}, there exists $V_i$ and $V_j$ that are transverse. In particular $d_{V_i}(H,\rho_{V_i}^{V_j}) \leq E\nu +2E$ and $d_{V_j}(H,\rho_{V_j}^{V_i}) \leq E\nu +2E$.  Relabeling $V_i =V_0$ and $V_j = V_1$, we satisfy the hypothesis of Lemma \ref{lem:pass-up_to_infinite} and therefore have a sequence of  elements $(z_i)_{i=0}^\infty \subseteq H$ and a sequence of domain $(U_i)_{i=1}^\infty$ so that each $U_i$ is in the $H$--orbit of ether $V_0$ or $V_1$ and $d_{U_j}(z_0,z_i) >8E$ whenever $j \leq i$ .
	
	Let $p = p(C)$ be the natural number from the Passing-up Lemma
	(Lemma \ref{lem:passing-up}).  Because $d_{U_j}(z_0,z_p) >8E$ for
	each $j \in \{1,\dots, p\}$, the Passing-up Lemma says there is a
	domain $W \in \mf{S}$ so that $d_W(z_0,z_p) >C$ and $U_j \propnest
	W$ for some $j \in \{1,\dots, p\}$.  There exists $h \in H$ so
	that $h U_j$ is equal to either $V_0$ or $V_1$.  However, this
	creates a contradiction with 
	\eqref{item:nest_max} as
	$d_{hW}(hz_0,hz_p)>C$ and $hU_j \propnest hW$.  There must
	therefore exist a constant $D\geq 0$ so that $\diam(\pi_V(H)) \leq
	D$ whenever $V$ is not nested into a eyrie of $H$.
\end{proof}

To use Lemma \ref{lem:bounded_proj_outside_eyries} and Theorem
\ref{thm:eyries} to prove that wide hierarchically quasiconvex
subgroups must have multiple eyries, we use the induced hierarchically
hyperbolic structure on a hierarchically quasiconvex subset shown in
\cite[Propostion 5.6]{BHS_HHSII}.  This construction applies to any
hierarchically quasiconvex subset, but we will describe it for
subgroups for simplicity.  A hierarchically quasiconvex subgroup $H$
of a $G$--HHS $(G,\mf{S})$ has an HHS structure $\mf{S}_H$ that is the
following restriction of $\mf{S}$ to $H$:
\begin{itemize}
	\item the index set for $\mf{S}_H$ is $\mf{S}$ and the relations are the same as in $\mf{S}$;
	\item the hyperbolic spaces for $\mf{S}_H$ are the convex hulls of the quasiconvex subsets $\pi_W(H)$;
	\item the projection maps are the restriction of the projection maps to $H$;
	\item for $V\trans W$ or $V \propnest W$, the relative projection from $V$ to $W$ in $\mf{S}_H$ is the closest point projection of $\rho_W^V$ onto $\pi_W(H)$.
\end{itemize}
The set $\mf{S}_H^\infty= \{V\in\mf{S}:\diam(\pi_V(H)) = 
\infty\}$ is precisely the set of unbounded domains for the HHS 
structure $\mf{S}_H$ when $H$ is hierarchically quasiconvex. Thus, the notation $\mf{S}_H^\infty$  is consistent with our past usage of the superscript $\infty$ to denote the set of unbounded domains in an HHS structure.

Using the above structure and Lemma
\ref{lem:bounded_proj_outside_eyries}, we establish that wide
hierarchically quasiconvex subgroups are characterized by having multiple
eyries.

\begin{prop}\label{prop:wide=mult_eyries}
	Let $(G,\mf{S})$ be a $G$--HHS and $H <G$ be an infinite, hierarchically quasiconvex subgroup.
	\begin{enumerate}
		\item If $H$ has a single eyrie, then $H$ is not wide as it is either virtually $\mathbb{Z}$ or is acylindrically hyperbolic.
		\item If $H$ has multiple eyries, then $H$ is wide, and, 
		moreover, it is quasi-isometric to the product of two infinite diameter quasi-geodesic spaces.
	\end{enumerate}
\end{prop}

\begin{proof}
	Assume first that $H$ has a single eyrie $W$.  By Lemma
	\ref{lem:bounded_proj_outside_eyries}, there is a number $D \geq
	0$ bounding the diameter of $\pi_V(H)$ for each $V\in \mf{S}$ not
	nested into $W$.  This implies that $H$ has an HHS structure with
	index set $\mf{S}_H \cap \mf{S}_W$ and not just all of $\mf{S}_H$;
	that is, we can remove all the domains from $\mf{S}_H$ that are
	not nested into $W$ without violating any of the HHS axioms.
	Importantly, the $\nest$--maximal domain of the structure
	$(H,\mf{S}_H \cap \mf{S}_W)$ is $W$.  Thus, \cite[Theorem
	14.3]{BHS_HHSI} says that $H$ acts acylindrically on the
	hyperbolic space associated to $W$ in $\mf{S}_H \cap \mf{S}_W$.
	As this space has infinite diameter by Theorem \ref{thm:eyries},
	this implies $H$ is either virtually cyclic or acylindrically
	hyperbolic \cite[Corollary 14.4]{BHS_HHSI}.  Either of these imply
	$H$ is not wide by \cite[Theorem
	1]{Sisto_quasiconveity_of_hyperbolically_embedded}.
	
	Now assume $H$ has multiple eyries $W_1,\dots, W_n$ with $n \geq 2$. Let $\P_{W_1}$ be the product region for $W_1$ in the HHS $(H,\mf{S}_H)$. By Theorem \ref{thm:eyries}, $W_1$ and $W_2$ are both unbounded domains in $\mf{S}_H$ and $W_1\perp W_2$. Thus $\P_{W_1}$ is quasi-isometric to the product of two infinite diameter quasi-geodesic spaces by Proposition \ref{prop:product_regions}\eqref{P_prop:wide}. Because there is a bounded $D \geq 0$ on the diameter of $\pi_V(H)$ for each $V \in \mf{S}$ that is not nested into one of the $W_i$, the distance formula for hierarchically hyperbolic spaces \cite[Theorem 4.5]{BHS_HHSII} says $H$ is quasi-isometric to $\P_{W_1}$ and hence wide.
\end{proof}

Combining Proposition \ref{prop:wide=mult_eyries} with Theorem \ref{thm:eyries} yields our proof of Theorem \ref{thm:wide_iff_join}.

\begin{proof}[Proof of Theorem \ref{thm:wide_iff_join}]
	Let $H$ be an infinite, hierarchically quasiconvex subgroup of a $G$--HHS $(G,\mf{S})$. We want to show that $H$ is wide if and only if its limit set $\Lambda(H)$ in $\partial_\Delta(G,\mf{S})$ is a non-trivial join. In the case when $H = G$, this implies $G$ is wide if and only if $\partial_\Delta(G,\mf{S})$ is a non-trivial join.
	
	Let $W_1,\dots, W_n$ be the eyries of $H$. If $V_p$ is the single domain in the support of a vertex $p \in \Lambda(H)$, then $\diam(\pi_{V_p}(H)) =\infty$. Thus, Theorem \ref{thm:eyries} says  $V_p \nest W_1$ or $V_p \perp W_1$. Since edges in $\partial_\Delta(G,\mf{S})$ correspond to orthogonality, this implies  $\Lambda(H)$ is a non-trivial join if and only if $n\geq 2$ (the two sides of the join are all vertices with support nested into $W_1$ and all vertices with support orthogonal to $W_1$).  By Proposition \ref{prop:wide=mult_eyries}, $n \geq 2$ if and only if $H$ is wide.
\end{proof}

\subsection{Thick of order 1}\label{subsec:thick_order_1}

We now turn characterize $G$--HHSs that are thick of order 1.
\begin{thm}\label{thm:thick_case}
	Let $(G,\frak S)$ be a $G$--HHS. 
	
	If $G$ is 
	 thick of order 1 relative to a collection
	of hierarchically quasiconvex wide subgroups, then $\partial_{\Delta}(G,\mf{S})$ is
	disconnected and contains a positive-dimensional $G$--invariant
	connected component.
	
	Conversely, if $\partial_{\Delta}(G,\mf{S})$ is disconnected and
	contains a positive-dimensional $G$--invariant connected component,
	then $G$ is  thick of order $1$ relative to a set of wide hierarchically quasiconvex
	subsets.
\end{thm}

\begin{proof}	
	Suppose first $G$ is  thick of order 1 relative to
	a collection $\{H_1,\dots, H_n\}$ of hierarchically quasiconvex
	wide subgroups.  As  $G$ is not wide,  $G$  has exactly one eyrie $W$ by Proposition \ref{prop:wide=mult_eyries}. Theorem \ref{thm:eyries} says that $\mc{C}W$ is  infinite diameter and no unbounded domain of $\mf{S}$ is orthogonal to $W$. Hence, $\partial_\Delta(G,\s)$ is
	disconnected as the points in $\partial \mc{C}W$ give isolated vertices of $\partial_\Delta(G,\mf{S})$.	
	
	Since each $H_i$ is hierarchically quasiconvex and wide, 
	Theorem \ref{thm:wide_iff_join} say the limit set, $\Lambda(H_i)$, of each $H_i$ is a non-trivial join.
	In particular, each $\Lambda(H_i)$ has positive dimension.

	Let $\Omega=\bigcup_{i=1}^n G\cdot \Lambda(H_i)$.  As $\Omega$ is
	a positive dimensional, $G$--invariant subset of $\partial_\Delta(G,\s)$, it remains to show that $\Omega$ is connected.  Fix
	points $\xi$ and $\zeta$ in $\Omega$.  We will exhibit a path in $\Omega$ from $\xi$ to $\zeta$. 
	
	 We have $\xi\in g\Lambda(H_i)$ and
	$\zeta\in g'\Lambda(H_j)$ for some $1\leq i,j\leq n$ and $g,g' \in G$.  Since $G$ is
	thick of order one relative to $\{H_1,\dots,H_n\}$, there is a constant $C \geq 0$ and
	sequence $g_0H_{i_0}=gH_i,g_1H_{i_1},\dots, g_{r-1}H_{i_{r-1}},
	g_rH_{i_r}=g'H_j$ so that $\mc{N}_C(g_kH_{i_k})\cap \mc{N}_C(g_{k+1}H_{i_{k+1}})$ 
	has infinite diameter for each $k=0,\dots, r-1$. So, for each
	such $k$, there is a sequence of points $(y^k_s)^{\infty}_{s=1}$ in $ \mc{N}_C(g_kH_{i_k})\cap \mc{N}_C(g_{k+1}H_{i_{k+1}})$ that limits to
	a point $\eta_k\in \Lambda (\mc{N}_C(g_kH_k))\cap \Lambda( \mc{N}_C(g_{k+1}H_{k+1}))$. By Lemma \ref{lem:bounded_difference_convergence}, $$\Lambda (\mc{N}_C(g_kH_k))\cap \Lambda( \mc{N}_C(g_{k+1}H_{k+1})) = g_k \Lambda(H_k) \cap g_{k+1}\Lambda(H_{k+1}).$$ 
	Each $g_k\Lambda(H_k)$ is connected as it is a non-trivial join. Hence,
	there is a path contained in $g_k\Lambda(H_k)$ from $\eta_{k}$
	to $\eta_{k+1}$.  The concatenation of these paths is a path from
	$\eta_0=\xi$ to $\eta_r=\zeta$ which is contained in $\Omega$, as
	desired.    Since $\Omega$ is a connected, positive dimensional, $G$--invariant subset of $\partial_\Delta( G,\s)$, it is contained in a positive dimensional, $G$--invariant connected component of $\partial_\Delta( G,\s)$.
	
	We now turn our attention to the backwards direction, and assume
	that $\partial_\Delta(G,\mf{S})$ is disconnected and contains a
	positive-dimensional $G$--invariant connected component $\Omega$. Now $\partial_\Delta(G,\mf{S})$  cannot be a join as it is disconnected. Thus, $G$ is not wide by Theorem \ref{thm:wide_iff_join}.
	
	 Each vertex $\xi\in \Omega^{(0)}$ is a point in $\partial
	\mathcal C U_{\xi}$ for some $U_{\xi}\in \s$.  Let
	\[
	\mc{P} =\{\mathbf P_{U_{\xi}}: \xi\in \Omega^{(0)}\}.
	\]  
	We claim that $G$ is thick of order one with respect to the 
	collection of subspaces $\mc{P}$. Note that the elements of 
	$\mc{P}$ are each uniformly hierarchically quasiconvex by 
	Proposition~\ref{prop:product_regions}.
	
	To see that each $\P_{U_\xi} \in \mc{P}$ is wide, observe that $\Omega$ being connected means the vertex $\xi$ is joined by an edge to a vertex $\zeta \in \Omega$. This implies $U_\xi \perp U_\zeta$. Since
	$\partial \mc C U_\xi$ and $\partial \mc C U_\zeta$ are non-empty, the domains $U_\xi$ and $U_\zeta$ are both unbounded. Thus, $\P_{U_{\xi}}$ is wide by Proposition \ref{prop:product_regions}\eqref{P_prop:wide}.
	
	That $\mc{P}$ satisfies the Thick Chains condition is a consequence of the fact that $\Omega$ is connected and the following claim: whenever $\xi,\zeta \in \Omega^{(0)}$ are joined by an edge,  the intersection $\mathbf
	P_{U_\xi}\cap\mathbf P_{U_\zeta}$ has infinite diameter. To prove this claim, choose sequences of points $(x_n)_{n=1}^\infty$ 		 in $\mc CU_\xi$ and $(y_n)^\infty_{n=1}$ in $\mc CU_\zeta$
	such that $d_{U_\xi}(x_0,x_n)>n$ and $d_{U_\zeta}(y_0,y_n)>n$.
	Applying the partial realization and uniqueness axioms
	(Definition~\ref{defn:HHS}\eqref{axiom:partial_realisation}\eqref{axiom:uniqueness}) to
	the pairs $\{x_0,y_0\}$ and $\{x_n,y_n\}$ yields points $p_0,p_n
	\in G$ so that $d_G(p_0,p_n)\to\infty$ as $n\to\infty$.  Moreover,
	partial realization ensures that $p_0,p_n\in \mathbf P_{U_\xi}\cap\mathbf
	P_{U_\zeta}$, which proves the claim.

	It remains to show that there exists $C\geq 0$ such that
	$G=\bigcup_{\xi \in \Omega^{(0)}}\mathcal N_C(\P_{U_\xi})$.  Let $g\in G$, fix
	any $\xi\in \Omega^{(0)}$, and let $h\in \mathbf P_{U_\xi}$.
	Then $g\in gh^{-1}\mathbf P_{U_\xi}=\mathbf P_{U_{gh^{-1}\xi}}$.
	Since $\Omega$ is $G$--invariant, $gh^{-1}\xi\in \Omega^{(0)}$, and the statement holds with $C=0$.  This completes
	the proof of thickness.
\end{proof}

The thick structure in the converse direction of Theorem 
\ref{thm:thick_case} consists of product regions in the $G$--HHS. We can therefore use Proposition \ref{prop:thickimpliesstronglythick} to state natural conditions where the converse direction can be promoted to strong thickness. 

\begin{cor}\label{cor:strong_thickness}
		Let $(G,\frak S)$ be a $G$--HHS and suppose $\partial_{\Delta}(G,\mf{S})$ is disconnected and
		contains a positive-dimensional $G$--invariant connected component $\Omega$. Let $\mf{U} \subseteq \mf{S}^\infty$ be the minimal $G$--invariant subset of domains so every point in $\Omega$ has support contained in $\mf{U}$.  If either
		\begin{enumerate}
			\item every closed ball in $G$ intersects at most finitely many elements of $\{ \P_U : U \in\mf{U}\}$, or
			\item the action of $G$ on $\mf{U} \times \mf{U}$  has finitely many orbits,
		\end{enumerate}
		 then $G$ will be strongly thick of order 1 relative to hierarchically quasiconvex subsets. In particular, $G$ will have quadratic divergence.
\end{cor}

\begin{proof}
	The proof of Theorem \ref{thm:thick_case} established that $G$  is thick relative to $\{\P_U :U \in \mf{U}\}$. The bulleted assumptions from Proposition \ref{prop:thickimpliesstronglythick} hold as follows:
	\begin{itemize}
		\item Hierarchically quasiconvex subsets are all quasiconvex by Proposition \ref{prop:HQC_and_hp}.
		\item Proposition \ref{prop:HQC_and_hp} also implies the intersection of two hierarchically quasiconvex subsets are hierarchically quasiconvex, and  hence coarsely connected.
		\item The definition of a $G$--HHS ensures that  $g\P_W = \P_{gW}$ for all  $g\in G$ and $W\in \mf{S}$. Since $\mf{U}$ is $G$--invariant,  $\{\P_U :U \in \mf{U}\}$ is  $G$--invariant.
	\end{itemize}
	Thus, Corollary \ref{cor:strong_thickness} is just a special case of  Proposition \ref{prop:thickimpliesstronglythick}.
	
	The result about quadratic divergence now follows immediately, 
	as Behrstock and Dru\c{t}u showed that strongly thick of order $k$ implies the divergence is at least $n^{k+1}$ \cite[Corollary 4.17]{BehrstockDrutu:thick2}.
\end{proof}

\begin{rem}
	The converse direction of Theorem \ref{thm:thick_case}---and hence
	Corollary \ref{cor:strong_thickness}---do not need the full power
	of a $G$--HHS. The given proofs hold for any HHS $(\mc{X},\mf{S})$
	with a cobounded action by a group $G$ so that $G$ also acts on
	$\mf{S}$ by relation-preserving bijections that satisfy the
	equivariance properties of Definition \ref{defn:nearlyHHG}.  An
	example where this occurs is the action of the mapping class group
	on the pants graph of the surface.
\end{rem}

\bibliography{HHS_Boundary}{}
\bibliographystyle{alpha}
\end{document}